\documentclass{article}

\usepackage{amsthm}
\usepackage{amsmath}
\usepackage{amssymb}
\usepackage{mathrsfs}
\usepackage{dsfont}
\usepackage[all]{xy}
\usepackage{tikz-cd}
\usepackage{adjustbox}
\usepackage{enumitem}
\usepackage{newclude}
\usepackage{hyperref}

\newcommand{\yo}{\text{\usefont{U}{min}{m}{n}\symbol{'210}}}

\DeclareFontFamily{U}{min}{}
\DeclareFontShape{U}{min}{m}{n}{<-> udmj30}{}

\theoremstyle{definition}
\newtheorem{Definition}{Definition}[subsection]

\theoremstyle{plain}
\newtheorem{Theorem}[Definition]{Theorem}

\theoremstyle{plain}
\newtheorem{Proposition}[Definition]{Proposition}

\theoremstyle{plain}
\newtheorem{Lemma}[Definition]{Lemma}

\theoremstyle{plain}
\newtheorem{Corollary}[Definition]{Corollary}

\theoremstyle{definition}
\newtheorem{Example}[Definition]{Example}

\theoremstyle{definition}
\newtheorem{Question}[Definition]{Question}

\theoremstyle{remark}
\newtheorem{Remark}[Definition]{Remark}

\theoremstyle{plain}
\newcommand{\thistheoremname}{}
\newtheorem*{genericthm*}{\thistheoremname}
\newenvironment{namedthm*}[1]
  {\renewcommand{\thistheoremname}{#1}%
   \begin{genericthm*}}
  {\end{genericthm*}}

\title{Compact Semisimple 2-Categories}
\author{Thibault D. Décoppet}

\begin{document}

\bibliographystyle{alpha}

     \maketitle
    \hspace{1cm}
    \begin{abstract}
        Working over an arbitrary field, we define compact semisimple 2-categories, and show that every compact semisimple 2-category is equivalent to the 2-category of separable module 1-categories over a finite semisimple tensor 1-category. Then, we prove that, over an algebraically closed field or a real closed field, compact semisimple 2-categories are finite. Finally, we explain how a number of key results in the theory of finite semisimple 2-categories over an algebraically closed field of characteristic zero can be generalized to compact semisimple 2-categories.
    \end{abstract}
    
\tableofcontents

\section*{Introduction}
\addcontentsline{toc}{section}{Introduction}

\subsection*{Overview}
\addcontentsline{toc}{subsection}{Overview}

The theory of fusion 1-categories over an algebraically closed field of characteristic zero is very well-understood (for instance, see \cite{EGNO}). Over an algebraically closed field of positive characteristic, the theory becomes more complicated. One example of this complication is the fact that the Drinfel'd center of an arbitrary fusion 1-category is no longer necessarily semisimple (see \cite{DSPS13}). In order to remedy this problem, one can define a separable fusion 1-category to be a fusion 1-category whose Drinfel'd center is semisimple. As explained in \cite{DSPS13}, this definition makes sense over any perfect field, and this class of fusion 1-categories is particularly well-behaved (see \cite{ENO} and \cite{Et}). Moreover, separable tensor 1-categories over a perfect field are fully dualizable objects of an appropriate symmetric monoidal 3-category $\mathbf{TC}^{ss}$ of finite semisimple tensor 1-categories (see \cite{DSPS13}). A key ingredient in the definition of this 3-category is the notion of separable module 1-category, a module 1-category whose 1-category of module endomorphisms is semisimple; this definition of a separable module 1-category makes sense over any field. Now, over an arbitrary perfect field, the $Hom$-2-categories of $\mathbf{TC}^{ss}$ have received little attention. This article aims to fill this gap, by studying the 2-categories of separable module 1-categories over finite semisimple tensor 1-categories.

Taking a slightly different perspective, this article should also be viewed as generalizing the theory of finite semisimple 2-categories introduced in \cite{DR} over an algebraically closed field of characteristic zero. In particular, we wish to categorify the notion of a finite semisimple 1-category over an arbitrary field. At first glance, one might expect the definition of finite semisimple semisimple 2-category to translate directly to arbitrary fields. However, this is not the case. More precisely, we certainly expect the 2-categories of separable module 1-categories over finite semisimple tensor 1-categories to give us examples of the desired categorified notion. But, as we will see, these 2-categories are rarely finite. This explains why we introduce the notion of a compact semisimple 2-category, which relaxes slightly the finiteness assumptions to accommodate our desired examples. More precisely, the existence of non-zero 1-morphisms organizes the set of simple objects of a semisimple 2-categories into ``connected components", and we say that a semisimple 2-category is compact if it has only finitely many ``connected components".

Perhaps the most important reason for introducing compact semisimple 2-categories is that this brings us closer to an abstract definition of finite tensor 2-categories. Namely, due to the existence of non-separable finite semisimple tensor 1-categories, the theory of finite semisimple tensor 1-categories over a field of characteristic $p$ is indissociable from the broader theory of finite tensor 1-categories (see \cite{EGNO}). Given that compact semisimple 2-categories categorify finite semisimple 1-categories over an arbitrary field, we can define compact semisimple tensor 2-categories by categorifying the notion of a finite semisimple tensor 1-category over an arbitrary field. Now, it is conjectured (see \cite{DR} and \cite{JFR}) that fusion 2-categories over an algebraically closed field of characteristic zero are fully dualizable objects of an appropriate symmetric monoidal 4-category. Some progress in this direction has been made (see \cite{D3}, \cite{D4}, and \cite{D6}). As compact semisimple tensor 2-categories generalize fusion 2-categories, and, by analogy with the decategorified situation recalled above, and described in details in \cite{DSPS13}, we expect that over a perfect field certain compact semisimple tensor 2-categories satisfying an additional separability axiom are fully dualizable objects of an appropriate symmetric monoidal 4-category generalizing the main result of \cite{BJS} (see also \cite{BJSS}). As a consequence, we could then use the cobordism hypothesis (see \cite{L}) to define interesting 4-dimensional TQFT's. Finally, over algebraically closed fields and real closed fields, we will see that compact semisimple 2-categories are in fact finite. Therefore, over such fields, compact semisimple tensor 2-categories are really fusion 2-categories. It is therefore natural to attempt to generalize to these more general fields the state-sum invariant described in \cite{DR}, which takes as input a fusion 2-category over an algebraically closed field of characteristic zero.

\subsection*{Results}
\addcontentsline{toc}{subsection}{Results}

We begin section \ref{sec:compact} by adapting to an arbitrary field $\mathds{k}$ the definition of a semisimple 2-category given in \cite{DR} for an algebraically closed field. We also show that the fundamental results derived in \cite{DR} hold in our more general context. Then, we show that given a finite semisimple tensor 1-category $\mathcal{C}$, its 2-category of separable right $\mathcal{C}$-module 1-categories $\mathbf{Mod}(\mathcal{C})$ is semisimple. After having done so, we abstract the properties of the 2-category $\mathbf{Mod}(\mathcal{C})$ in order to obtain the definition of a compact semisimple 2-category. Namely, given a semisimple 2-category $\mathfrak{C}$, we can define $\pi_0(\mathfrak{C})$ as the quotient of the set of simple objects of $\mathfrak{C}$ under the relation identifying two simple objects if there is a non-zero 1-morphism between them. A semisimple 2-category $\mathfrak{C}$ is called compact if $\pi_0(\mathfrak{C})$ is a finite set. Further, we prove that compact semisimple 2-categories correspond to finite semisimple tensor 1-categories in the following sense.

\begin{namedthm*}{Theorem \ref{thm:semisimplemodule}}
Let $\mathfrak{C}$ be a compact semisimple 2-category. Then, there exists a finite semisimple tensor 1-category $\mathcal{C}$ such that $\mathfrak{C}\simeq\mathbf{Mod}(\mathcal{C})$.
\end{namedthm*}

\noindent We end the first section by introducing the notion of a locally separable compact semisimple 2-category, provided $\mathds{k}$ is perfect, which abstracts the features of the semisimple 2-category $\mathbf{Mod}(\mathcal{C})$ when $\mathcal{C}$ is a separable tensor 1-category.

In section \ref{sec:finite}, we investigate the conditions that are necessary to ensure that a compact semisimple 2-category is finite. We begin by showing that over certain base fields, this can never be the case.

\begin{namedthm*}{Proposition \ref{prop:infiniteseparableclosure}}
Let $\mathds{k}$ be a field whose separable closure $\mathds{k}^{sep}$ is an infinite extension. Then, given any finite semisimple tensor 1-category $\mathcal{C}$ over $\mathds{k}$, the compact semisimple 2-category $\mathbf{Mod}(\mathcal{C})$ has infinitely many equivalence classes of simple objects.
\end{namedthm*}

\noindent Now, it happens that, if $\mathds{k}$ is perfect, requiring that $\mathds{k}^{sep}$ is a finite extension of $\mathds{k}$, forces $\mathds{k}$ to be either an algebraically closed field or a real closed field. In this case, we can prove that all compact semisimple 2-categories are finite.

\begin{namedthm*}{Theorem \ref{thm:modulefinitesemisimple}}
Let $\mathcal{C}$ be a multifusion 1-category over an algebraically closed field or a real closed field. The 2-category $\mathbf{Mod}(\mathcal{C})$ of separable right $\mathcal{C}$-module 1-categories is a finite semisimple 2-category.
\end{namedthm*}

\begin{namedthm*}{Corollary \ref{cor:separablefinite}}
Over an algebraically closed field or a real closed field, a semisimple 2-category is finite if and only if it is compact.
\end{namedthm*}

\noindent The underlying idea of our proof comes from section 9.1 of \cite{EGNO}, where these results are established over an algebraically closed field of characteristic zero. Namely, our proof begins by translating the above problem to a question about the deformations of suitable monoidal functors between finite semisimple 1-categories. Then, a careful analysis of the vanishing of Davydov-Yetter cohomology combined with various results on the openness of orbits under an algebraic group action shows that these objects do not admit any non-trivial deformations, giving the desired result.

Further, the same ideas can be used to show that separable multifusion 1-categories do not admit non-trivial deformations. Our original contribution is the case of multifusion 1-categories over real closed fields.

\begin{namedthm*}{Theorem \ref{thm:nodeformationmulti}}
Let $\mathcal{C}$ be a separable multifusion 1-category over an algebraically closed field or a real closed field. Then, $\mathcal{C}$ has no non-trivial deformations. Consequently, the number of equivalence classes of separable multifusion 1-categories with a fixed Grothendieck ring is finite.
\end{namedthm*}

\noindent Over an algebraically closed field of characteristic zero, the above result was first proven in \cite{ENO}. Moreover, over an algebraically closed field of positive characteristic, the statement is given in section 9 of \cite{ENO}, and a proof can be obtained using the techniques developed in their article. The second section then ends by explaining how to use our techniques to analyze the field of definition of a separable multifusion 1-category over an algebraically closed field or a real closed field. In details, it is shown in \cite{EGNO} that a multifusion 1-category over an algebraically closed field of characteristic zero is defined over $\overline{\mathbb{Q}}$. We show that, over an algebraically closed field of positive characteristic $p$, separable multifusion 1-categories are always defined over $\overline{\mathbb{F}}_p$. Further, over a real closed field, separable multifusion 1-categories are defined over $\mathbb{R}^{alg}=\overline{\mathbb{Q}}\cap\mathbb{R}$.

Finally, section \ref{sec:additional} begins by showing that the main result of \cite{D1} can be generalized to compact semisimple 2-categories over arbitrary fields. Namely, if $\mathbf{TC}^{rss}$ denotes the 3-category of finite semisimple tensor 1-categories and right separable bimodule 1-categories between them, and $\mathbf{CSS2C}$ denotes the 3-category of compact semisimple 2-categories, we have the following theorem.

\begin{namedthm*}{Theorem \ref{thm:equivalence}}
Over an arbitrary field $\mathds{k}$, there is an equivalence of linear 3-categories $$\mathbf{Mod}(-):\mathbf{TC}^{rsep}\xrightarrow{\simeq} \mathbf{CSS2C}.$$
\end{namedthm*}

\noindent Provided $\mathds{k}$ is perfect, the above equivalence of 3-categories restricts to an equivalence between the full sub-3-category of $\mathbf{TC}^{rsep}$ on the separable tensor 1-categories, and the full sub-3-category of $\mathbf{CSS2C}$ on the locally separable compact semisimple 2-categories. In fact, locally separable compact semisimple 2-categories enjoy additional properties. In particular, we show that the Yoneda embedding of a locally separable compact semisimple 2-category $\mathfrak{C}$ is an equivalence, a result which was shown in \cite{DR} over an algebraically closed field of characteristic zero. More precisely, we show that the canonical embedding of $\mathfrak{C}$ into the 2-category of linear $\mathbf{2Vect}$-valued presheaves on $\mathfrak{C}$ is an equivalence. We move on to give the definition of a compact semisimple tensor 2-category over an arbitrary field, and explore its fundamental properties. We conclude by analyzing the example of the 2-category of separable right module 1-categories over the fusion 1-category of $\mathbb{Z}/p$-graded vector spaces over an arbitrary algebraically closed field.

\subsubsection*{Acknowledgments}

I am grateful to Christopher Douglas his supervison and the extensive feedback he provided when preparing this article. I would also like to thank Andr\'e Henriques and Christoph Weis for spotting a mistake in a draft of this paper. Finally, I am particularly indebted to the person who referred the article \cite{D3}, and whose comments prompted me to write this article in the first place.

\section{Compact Semisimple 2-Categories}\label{sec:compact}

Let $\mathds{k}$ be an arbitrary field. All of the categories we will consider are $\mathds{k}$-linear.

\subsection{Semisimple 2-Categories}

\begin{Definition}
Recall the following standard definitions:
\begin{itemize}
    \item An object of an abelian 1-category is called simple if it is not zero, and it has no nontrivial subobject. The endomorphism space of a simple object is a division algebra over $\mathds{k}$.
    \item An abelian 1-category is called semisimple provided that every object splits as a finite direct sum of simple objects.
    \item An abelian 1-category is called finite it is is equivalent to the category of finite-dimensional modules over a finite-dimensional $\mathds{k}$-algebra.
\end{itemize}
\end{Definition}

\begin{Remark}
There is an intrinsic characterization of finite 1-categories (see \cite{DSPS14}). In particular, such a 1-category has finitely many equivalence classes of simple objects, and finite-dimensional $Hom$-spaces.
\end{Remark}

\begin{Definition}
A semisimple 2-category is a locally semisimple $\mathds{k}$-linear 2-category that has right and left adjoints for 1-morphisms and that is Cauchy complete (see definition 1.2.1 of \cite{D1}).
\end{Definition}

\begin{Definition}\label{def:semisimple}
An object $X$ of a semisimple 2-category $\mathfrak{C}$ is called simple if its identity 1-morphism $Id_X$ is a simple object of $End_{\mathfrak{C}}(X)$.
\end{Definition}

The categorical Schur lemma (proposition 1.2.19 of \cite{DR}) continues to hold in our more general context.

\begin{Lemma}\label{lem:catschur}
Let $f:X\rightarrow Y$, and $g:Y\rightarrow Z$ be two nonzero 1-morphisms between simple object of a semisimple 2-category. Then, $g\circ f$ is nonzero.
\end{Lemma}
\begin{proof}
Let $g^*:Z\rightarrow Y$ be a right adjoint to $g$, and write $\epsilon:g^*\circ g \Rightarrow Id_Y$ for the counit. As $g$ is non-zero and $Id_Y$ is simple, there exists a section $\gamma$ of $\epsilon$. In particular, the composite 2-morphism $(\gamma \circ f)\cdot (\epsilon \circ f)$ is the identity on $f$, and factors through $g^*\circ g\circ f$. As $f$ is non-zero, this shows that $g\circ f$ is non-zero.
\end{proof}

In fact, it turns out that most of the theory of semisimple 2-categories developed in \cite{DR} holds over an arbitrary field. Let us make explicit the following results.

\begin{Lemma}\label{lem:decompositionprimer}
Let $C$ be an object of a semisimple 2-category $\mathfrak{C}$. Then, a decomposition $Id_C\cong \oplus_{i=1}^nf_i$ of $Id_C$ into a direct sum of non-zero 1-morphisms induces a decomposition $C\simeq \boxplus_{i=1}^nC_i$ of $C$ into a direct sum of non-zero objects $C_i$ of $\mathfrak{C}$.
\end{Lemma}
\begin{proof}
Observe that by theorem 3.1.7 of \cite{GJF}, a locally semisimple 2-category with adjoints is Cauchy complete if and only if it has direct sums and every separable monad as defined in \cite{DR} splits. Thus, our definition \ref{def:semisimple} is equivalent to definition 1.4.1 of \cite{DR} over an arbitrary field. Thus, the proof of proposition 1.3.16 of \cite{DR} applies in our context.
\end{proof}

\begin{Proposition}\label{prop:decomposition}
In a semisimple 2-category $\mathfrak{C}$, every object splits as a direct sum of simple objects. Further, this splitting is unique up to permutation and equivalence.
\end{Proposition}
\begin{proof}
Let $C$ be an object of $\mathfrak{C}$. Then, as $End_{\mathfrak{C}}(C)$ is a semisimple 1-category, we can pick $Id_C\cong \oplus_{i=1}^nf_i$ a decomposition of $Id_C$ into a direct sum of simple objects of $End_{\mathfrak{C}}(C)$. Lemma \ref{lem:decompositionprimer} then yields the first part of the result. To get the second part, observe that, in fact, splittings $C\simeq \boxplus_{i=1}^nC_i$ of $C$ into a direct sum of simple objects $C_i$ up to equivalence, correspond to splittings $Id_C\cong \oplus_{i=1}^nf_i$ of $Id_C$ into a direct sum of simple objects of $End_{\mathfrak{C}}(C)$. Thence, the second part of the result follows from the uniqueness of the splitting of an object as a direct of simple objects in a semisimple 1-category.
\end{proof}

Finally, over algebraically closed fields of characteristic zero, the following objects play a central role in \cite{DR}. We shall say more about their existence over arbitrary fields in section \ref{sec:finite}.

\begin{Definition}\label{def:finitesemisimple}
A semisimple 2-category is finite if it is locally finite (i.e. $Hom$-categories are finite 1-categories), and has finitely many equivalence classes of simple objects.
\end{Definition}

\subsection{2-Categories of Separable Module 1-Categories}\label{sub:sepmod}

Recall the following definitions from \cite{DSPS13}:
\begin{itemize}
    \item A tensor 1-category is a linear 1-category equipped with a rigid monoidal structure.
    \item Over an algebraically closed field or a real closed field, a multifusion 1-category is a finite semisimple tensor 1-category. A fusion 1-category is a multifusion 1-category whose monoidal unit is a simple object.
    \item Let $\mathcal{C}$ be a monoidal 1-category. A right $\mathcal{C}$-module 1-category is a 1-category $\mathcal{M}$ equipped with a bilinear right action $\mathcal{M}\times \mathcal{C}\rightarrow \mathcal{M}$ that is suitably coherent (see \cite{DSPS13} definition 2.2.6).
\end{itemize}

Now, let us fix $\mathcal{C}$ a finite semisimple tensor 1-category. Given a finite semisimple right $\mathcal{C}$-module 1-category $\mathcal{M}$, Ostrik's theorem (in the generality of theorem 2.18 of \cite{DSPS14}) tells us that there is an algebra object $A$ in $\mathcal{C}$, and an equivalence of right $\mathcal{C}$-module 1-categories between $\mathcal{M}$ and $\mathrm{Mod}_{\mathcal{C}}(A)$ the 1-category of left $A$-modules in $\mathcal{C}$.

\begin{Definition}
A finite semisimple right $\mathcal{C}$-module 1-category $\mathcal{M}$ is called separable if the algebra $A$ is separable. We then say that $\mathcal{M}$ is a separable right $\mathcal{C}$-module 1-category.
\end{Definition}

\begin{Remark}
There is also an intrinsic characterization of separable right $\mathcal{C}$-module 1-categories. Namely, it is shown in theorem 1.5.4 of \cite{DSPS13} that a finite semisimple right $\mathcal{C}$-module 1-category $\mathcal{M}$ is separable if and only if the 1-category $\mathrm{Fun}_{\mathcal{C}}(\mathcal{M},\mathcal{M})$ of linear right $\mathcal{C}$-module functors $\mathcal{M}\rightarrow \mathcal{M}$ is finite semisimple.
\end{Remark}

\begin{Lemma}\label{lem:Homfinitesemisimple}
Let $\mathcal{M}$ and $\mathcal{N}$ be two separable right $\mathcal{C}$-module 1-categories. Then, $\mathrm{Fun}_{\mathcal{C}}(\mathcal{M},\mathcal{N})$ is finite semisimple.
\end{Lemma}
\begin{proof}
Let $A$, $B$ be separable algebra in $\mathcal{C}$ such that $\mathrm{Mod}_{\mathcal{C}}(A)\simeq \mathcal{M}$, and $\mathrm{Mod}_{\mathcal{C}}(B)\simeq \mathcal{N}$ as right $\mathcal{C}$-module 1-categories. Theorem 2.2.17, corollary 2.4.11, and corollary 2.4.17 of \cite{DSPS13} prove that $$\mathrm{Fun}_{\mathcal{C}}(\mathcal{M},\mathcal{N})\simeq \mathrm{Bimod}_{\mathcal{C}}(A,B).$$ Furthermore, the right hand-side is a finite category, so it only remains to prove that every object is projective. Let $M$ be an $A$-$B$-bimodule in $\mathcal{C}$. We write $m_A$ for the multiplication of $A$, and $\Delta_A$ for its splitting as a map of $A$-$A$-bimodules. Similarly, we write $m_B$ for the multiplication of $B$, and $\Delta_B$ for its splitting as a map of $B$-$B$-bimodules. The bimodule map $$A\otimes_A M\otimes_B B\xrightarrow{\Delta_A\otimes_A Id_M\otimes_B \Delta_B} (A\otimes A)\otimes_A M\otimes_B (B\otimes B)$$ is a splitting for the canonical bimodule map $$(A\otimes A)\otimes_A M\otimes_B (B\otimes B) \rightarrow A\otimes_A M\otimes_BB.$$ Thus, we find that $M$ is a direct summand of a free bimodule, which prove the claim.
\end{proof}

We are now ready to given our first examples of semisimple 2-categories.

\begin{Proposition}\label{prop:modulesemisimple}
The 2-category $\mathbf{Mod}(\mathcal{C})$ of separable right $\mathcal{C}$-module 1-categories is a semisimple 2-category that is locally finite.
\end{Proposition}
\begin{proof}
The proof follows using the argument given in the proof of theorem 1.4.8 of \cite{DR}. We give some key points for the reader's convenience. Firstly, observe that given two separable right $\mathcal{C}$-module 1-categories $\mathcal{M}$, and $\mathcal{N}$, then $Hom_{\mathbf{Mod}(\mathcal{C})}(\mathcal{M},\mathcal{N})$ is finite semisimple by lemma \ref{lem:Homfinitesemisimple}. The existence of adjoints for 1-morphisms follows from \cite{DSPS14}. Moreover, it is clear that $\mathbf{Mod}(\mathcal{C})$ has a zero object and direct sums for objects. Thus, it only remains to prove that every 2-condensation monad splits. By theorem 3.1.7 of \cite{GJF}, it is enough to show that separable algebras split. The argument used in the proof of theorem 1.4.8 of \cite{DR} also applies in our context, which concludes the proof.
\end{proof}

It is useful to describe the simple objects of $\mathbf{Mod}(\mathcal{C})$ more explicitly.

\begin{Lemma}
A separable right $\mathcal{C}$-module 1-category $\mathcal{M}$ is simple as an object of $\mathbf{Mod}(\mathcal{C})$ if and only if it is indecomposable.
\end{Lemma}
\begin{proof}
By definition, an object $\mathcal{M}$ of $\mathbf{Mod}(\mathcal{C})$ is simple if and only if the monoidal unit of $\mathrm{End}_{\mathbf{Mod}(\mathcal{C})}(\mathcal{M})$ is simple. But, this finite semisimple tensor 1-category is exactly $\mathrm{Fun}_{\mathcal{C}}(\mathcal{M},\mathcal{M})$, whose monoidal unit is simple if and only if $\mathcal{M}$ is indecomposable as a right $\mathcal{C}$-module 1-category.
\end{proof}

Proposition \ref{prop:modulesemisimple} gives us plenty of examples of semisimple 2-categories. Let us examine one in particular.

\begin{Example}\label{ex:infinitesimples}
Let $\mathds{k} = \mathbb{F}_p$ for a prime $p$, and $G$ a finite group. Write $\mathcal{C}$ for the finite semisimple tensor 1-category of $G$-graded finite-dimensional $\mathbb{F}_p$-vector spaces. By the above proposition, $\mathbf{Mod}(\mathcal{C})$ is a semisimple 2-category. Let us prove that it has infinitely many equivalence classes of simple objects. Given any positive integer $q$, write $\mathcal{C}_q$ for the finite semisimple 1-category of $G$-graded finite-dimensional $\mathbb{F}_{p^q}$-vector spaces. The 1-category $\mathcal{C}_q$ admits an obvious coherent right action by $\mathcal{C}$. Further, these 1-categories are pairwise non-equivalent. Therefore, it suffices to show that the right $\mathcal{C}$-module 1-categories $\mathcal{C}_q$ are separable. In order to see this, observe that $\mathbb{F}_{p^q}$ is a separable algebra in $\mathbf{Vect}_{\mathbb{F}_p}$ the finite semisimple tensor 1-category of finite dimensional $\mathbb{F}_p$-vector spaces (see corollary 4.5.8 of \cite{For}). Pushing this algebra forward along the monoidal inclusion $\mathbf{Vect}_{\mathbb{F}_p}\hookrightarrow \mathcal{C}$, we obtain a separable algebra in $\mathcal{C}$, which we also denote by $\mathbb{F}_{p^q}$. Now, it is not hard to show that $\mathrm{Mod}_{\mathcal{C}}(\mathbb{F}_{p^q})\simeq \mathcal{C}_q$ as right $\mathcal{C}$-module 1-categories, which proves the claim.
\end{Example}

\subsection{Definition \& First Properties}

Example \ref{ex:infinitesimples} shows that we can not expect $\mathbf{Mod}(\mathcal{C})$ to be a finite semisimple 2-category in general. Nevertheless, such 2-categories still have reasonable finiteness properties. In order to capture them, we need the following definition.

\begin{Definition}
Let $\mathfrak{C}$ be a semisimple 2-category. Two simple objects $X$ and $Y$ of $\mathfrak{C}$ are said to be connected if there is a non-zero 1-morphism between them.
\end{Definition}

As is explained in \cite{DR} over an algebraically closed field of characteristic zero, the categorical Schur lemma shows that being connected defines an equivalence relation on the set of simple objects of a given semisimple 2-category. But the categorical Schur lemma holds over an base field (see lemma \ref{lem:catschur}), so we can make the following definitions.

\begin{Definition}
The set $\pi_0(\mathfrak{C})$ is the quotient of the set of simple objects of $\mathfrak{C}$ under the equivalence relation of being connected. A semisimple 2-category $\mathfrak{C}$ is connected if $\pi_0(\mathfrak{C})$ is a singleton.
\end{Definition}

\begin{Definition}
A semisimple 2-category $\mathfrak{C}$ is compact if it is locally finite and $\pi_0(\mathfrak{C})$ is finite.
\end{Definition}

We now show that $\mathbf{Mod}(\mathcal{C})$ is a compact semisimple 2-category for every finite semisimple tensor 1-category $\mathcal{C}$.

\begin{Theorem}\label{thm:modulenearfinitesemisimple}
Let $\mathcal{C}$ be a finite semisimple tensor 1-category. The 2-category of separable right $\mathcal{C}$-module 1-categories is a compact semisimple 2-category.
\end{Theorem}
\begin{proof}
Let us begin by observing that for any nonzero separable right $\mathcal{C}$-module $\mathcal{M}$, there exists a right $\mathcal{C}$-module functor $F:\mathcal{C}\rightarrow \mathcal{M}$. Namely, pick an object $M$ in $\mathcal{M}$, and set $F(C)= M\otimes C$ with the canonical right $\mathcal{C}$-module structure. By proposition \ref{prop:decomposition}, we know that $\mathcal{C}$ splits as a direct sum of finitely many simple objects. Thus, we have shown that any simple object of $\mathbf{Mod}(\mathcal{C})$ admits a nonzero 1-morphism from a direct summand of $\mathcal{C}$. This finishes the proof.
\end{proof}

In fact, the converse also holds.

\begin{Definition}
We say that a semisimple 2-category $\mathfrak{C}$ has a generator if there exists $X$ in $\mathfrak{C}$ such that the inclusion $\mathrm{B}End_{\mathfrak{C}}(X)\hookrightarrow \mathfrak{C}$ is a Cauchy completion.
\end{Definition}

\begin{Theorem}\label{thm:semisimplemodule}
Let $\mathfrak{C}$ be a compact semisimple 2-category. Then, $\mathfrak{C}$ has a generator. In particular, there exists a finite semisimple tensor 1-category $\mathcal{C}$ such that $\mathfrak{C}\simeq \mathbf{Mod}(\mathcal{C})$.
\end{Theorem}
\begin{proof}
Let us pick a simple object $X_i$ in every equivalence class of $\pi_0(\mathfrak{C})$. We define $X:=\boxplus_i X_i$, which is an object of $\mathfrak{C}$, and write $\mathcal{C}:=End_{\mathfrak{C}}(X)$. Now, the inclusion $$\begin{tabular}{ccc}$\mathrm{B}\mathcal{C}$&$\hookrightarrow$&$\mathfrak{C}$\\ $*$&$\mapsto$&$X$\end{tabular}$$ is full on 1-morphisms and fully faithful on 2-morphisms. As the right hand-side is Cauchy complete, using definition 1.2.1 of \cite{D1}, we get a linear 2-functor $$F:\mathbf{Mod}(\mathcal{C})\simeq Cau(\mathrm{B}\mathcal{C})\hookrightarrow \mathfrak{C},$$ which is still full on 1-morphisms and fully faithful on 2-morphisms. Thence, it only remains to prove essential surjectivity. Given that $\mathbf{Mod}(\mathcal{C})$ has direct sums, it is enough to show that every simple object in $\mathfrak{C}$ is in the essential image of $F$. In order to do so, we prove that for any simple object $Y$ of $\mathfrak{C}$, there exists a 2-condensation monad on $X$, whose splitting is $Y$. Let $f:X\rightarrow Y$ be a simple 1-morphism, which exists by construction of $X$. We write $f^*$ for its right adjoint, $\epsilon^f:f\circ f^*\Rightarrow Id_Y$ for the corresponding counit. As $f$ is non-zero, and $Id_Y$ is simple, $\epsilon^f$ admits a section, which we denote by $\gamma$. It follows from the definitions that the data $(X,Y, f, f^*, \epsilon, \gamma)$ defines a 2-condensation. In particular, there exists a 2-condensation monad on $X$ whose splitting in $\mathfrak{C}$ is $Y$ as desired.
\end{proof}

\subsection{Locally Separable Compact semisimple 2-Categories}

We now assume that $\mathds{k}$ is a perfect field. Let us recall the following definition from \cite{DSPS13}.

\begin{Definition}
A finite semisimple tensor 1-category is called separable if it is separable as a right $\mathcal{C}^{\otimes op}\boxtimes \mathcal{C}$-module 1-category.
\end{Definition}

\begin{Remark}
It is shown in \cite{DSPS13} that a finite semisimple tensor 1-category is separable if and only if its Drinfel'd center is finite semisimple. Furthermore, they show that over a field of characteristic zero, every finite semisimple tensor 1-category is separable.
\end{Remark}

Let $\mathcal{C}$ be a finite semisimple tensor 1-category. Following \cite{EGNO}, a decomposition of the unit into simple summands induces an equivalence of finite semisimple tensor 1-categories $$\mathcal{C}\simeq \begin{pmatrix}
_{1}\mathcal{C}_1 & \cdots & _1\mathcal{C}_n\\
\vdots & \ddots & \vdots\\
_n\mathcal{C}_1&\cdots & _n\mathcal{C}_n\\
\end{pmatrix}$$
for some integer $n$, where the monoidal unit of $_i\mathcal{C}_i$ is simple for all $i$.

\begin{Lemma}\label{lem:multifusionseparability}
A finite semisimple tensor 1-category $\mathcal{C}$ is separable if and only if the finite semisimple tensor sub-1-categories $_i\mathcal{C}_i$ are separable for all $i$.
\end{Lemma}
\begin{proof}
Let us begin by assuming that $\mathcal{C}$ is connected, i.e. ${_{i}\mathcal{C}_j}$ is non-zero for every $i,j$. In this case, $\mathcal{C}$ is Morita equivalent to $_{1}\mathcal{C}_1$, so the statement follows from the fact that the Drinfel'd center is a Morita invariant. More generally, $\mathcal{C}$ is Morita equivalent to a finite semisimple tensor 1-category of the form $\oplus_{j\in J}{_{j}\mathcal{C}_j}$ for some subset $J\subseteq \{1,...,n\}$. It follows from the definitions that $\oplus_{j\in J}{_{j}\mathcal{C}_j}$ is separable if and only if each of the finite semisimple tensor 1-category ${_{j}\mathcal{C}_j}$ is.
\end{proof}

\begin{Definition}
A compact semisimple 2-category $\mathfrak{C}$ is locally separable if $End_{\mathfrak{C}}(X)$ is a separable finite semisimple tensor 1-category for every simple object $X$ of $\mathfrak{C}$.
\end{Definition}

\begin{Remark}
Thanks to lemma \ref{lem:multifusionseparability}, in order to check that a compact semisimple 2-category $\mathfrak{C}$ is locally separable, it is enough to prove that $End_{\mathfrak{C}}(X_i)$ is a separable finite semisimple tensor 1-category, where the $X_i$ are simple objects representing the equivalence classes in $\pi_0(\mathfrak{C}$.
\end{Remark}

\begin{Theorem}\label{thm:sepmultsep2cat}
Let $\mathcal{C}$ be a separable finite semisimple tensor 1-category. The 2-category of separable right $\mathcal{C}$-module 1-categories is a locally separable compact semisimple 2-category.
\end{Theorem}
\begin{proof}
Let $\mathcal{M}$ be an indecomposable separable right $\mathcal{C}$-module 1-category. Then, $End_{\mathcal{C}}(\mathcal{M})$ is Morita equivalent (via $\mathcal{M}$) to a finite semisimple tensor sub-1-category of $\mathcal{C}$. As all such finite semisimple tensor sub-1-categories are separable the statement follows.
\end{proof}

\begin{Theorem}\label{thm:sep2catsepmult}
Let $\mathfrak{C}$ be a locally separable compact semisimple 2-category. Then, there exists a separable finite semisimple tensor 1-category $\mathcal{C}$ such that $\mathfrak{C}\simeq \mathbf{Mod}(\mathcal{C})$.
\end{Theorem}
\begin{proof}
The argument used to prove theorem \ref{thm:semisimplemodule} applies, but, thanks to our hypothesis, the finite semisimple tensor 1-category we construct is separable.
\end{proof}

\begin{Remark}\label{rem:separableperfectfinitesemisimple}
As $\mathds{k}$ is a perfect field, and $\mathcal{C}$ is separable, $\mathbf{Mod}(\mathcal{C})$ may also be described as the 2-category of finite semisimple right $\mathcal{C}$-module 1-categories as can be seen using proposition 2.5.10 of \cite{DSPS13}.
\end{Remark}

\begin{Example}\label{ex:separablevectG}
Let $G$ be a finite group, and $\mathds{k}$ be an algebraically closed field such that $char(\mathds{k})\nmid |G|$. Define $\mathcal{C}$ to be the finite semisimple tensor 1-category of finite dimensional $G$-graded $\mathds{k}$-vector spaces. Then, $\mathcal{C}$ is separable. We now describe the structure of the compact semisimple 2-category $\mathbf{Mod}(\mathcal{C})$. Thanks to the separability of $\mathcal{C}$, we know that the equivalence classes simple objects of $\mathbf{Mod}(\mathcal{C})$ correspond to the equivalence classes of indecomposable finite semisimple right $\mathcal{C}$-module 1-categories. Following example 7.4.10 of \cite{EGNO}, we know that the latter are given by pairs consisting of a subgroup $H\subseteq G$ and a cohomology class $\phi\in H^2(H;\mathds{k}^{\times})$ (under a suitable equivalence relation induced by conjugation). Thus, there are only finitely many equivalences classes of simple objects in $\mathbf{Mod}(\mathcal{C})$, i.e. it is a finite semisimple 2-category.
\end{Example}

\section{Finite Semisimple 2-Categories}\label{sec:finite}

The objective of this section is to give examples of finite semisimple 2-categories. Our first task is to identify when the base field allows for such 2-categories to exist. Having done so, we prove that over an appropriate base field, every compact semisimple 2-category is finite.

\subsection{Non-examples}

Given $\mathcal{C}$ a finite semisimple tensor 1-category. Example \ref{ex:infinitesimples} suggests that there is, in general, no hope for the 2-category of separable right $\mathcal{C}$-module 1-categories to be finite semisimple if our base field has infinitely many non-equivalent finite separable extensions. We begin by making this precise.

\begin{Proposition}\label{prop:infiniteseparableclosure}
Let $\mathds{k}$ be a field whose separable closure $\mathds{k}^{sep}$ is an infinite extension. Then, given any finite semisimple tensor 1-category $\mathcal{C}$ over $\mathds{k}$, the compact semisimple 2-category $\mathbf{Mod}(\mathcal{C})$ has infinitely many equivalence classes of simple objects.
\end{Proposition}
\begin{proof}
Let $\mathcal{M}_0$ be any separable right $\mathcal{C}$-module 1-category. For instance, take $\mathcal{M}_0=\mathcal{C}$. We let $$N(\mathcal{M}_0):=\mathrm{max}\{\mathrm{dim}_{\mathds{k}}(End_{\mathcal{M}}(X))|X\in\mathcal{M},\ X\ \textrm{simple}\}\in\mathbb{N}.$$
We claim that there exists indecomposable separable right $\mathcal{C}$-module 1-categories $\mathcal{M}_i$ for every non-negative integer $i$, such that $N(\mathcal{M}_i)>N(\mathcal{M}_{i-1})$ for all $i\geq 1$. This readily implies the statement of the proposition.

In order to prove the above claim, it is enough to prove that given any indecomposable separable right $\mathcal{C}$-module 1-category $\mathcal{M}$, there exists an indecomposable separable right $\mathcal{C}$-module 1-category $\mathcal{M}'$ such that $N(\mathcal{M}')>N(\mathcal{M})$. So let us fix such an $\mathcal{M}$. As the extension $\mathds{k}^{sep}/\mathds{k}$ is infinite, there exists a finite separable extension $\mathds{k}'/\mathds{k}$ whose degree is at least $N(\mathcal{M})+1$. In particular, $\mathds{k}'$ is finite separable $\mathds{k}$-algebra.

We claim that the Deligne tensor product $\widehat{\mathcal{M}}:=\mathbf{Vect}_{\mathds{k}'}\boxtimes\mathcal{M}$ over $\mathds{k}$ is a separable right $\mathcal{C}$-module 1-category. Namely, by hypothesis, there exists a separable algebra $A$ in $\mathcal{C}$ such that $\mathcal{M}\simeq \mathrm{Mod}_{\mathcal{C}}(A)$ as right $\mathcal{C}$-module 1-categories. We then find that $$\mathbf{Vect}_{\mathds{k}'}\boxtimes\mathcal{M}\simeq \mathrm{Mod}_{\mathcal{C}}(\mathds{k}'\otimes A),$$ as right $\mathcal{C}$-module 1-categories. As $\mathds{k}'\otimes A$ is a separable algebra in $\mathcal{C}$, this proves the claim.

Note that $\widehat{\mathcal{M}}$ might not be indecomposable, so we pick an indecomposable summand $\mathcal{M}'$. It remains to check that $N(\mathcal{M}')\geq N+1$. Pick $X$ a simple object of $\mathcal{M}$. Abusing the notation, we will also write $X$ for the corresponding object of $\widehat{\mathcal{M}}$, which might not be simple. By construction, $$End_{\widehat{\mathcal{M}}}(X)\cong End_{\mathcal{M}}(X)\otimes_{\mathds{k}}\mathds{k}',$$ which is a finite semisimple $\mathds{k}$-algebra (see corollary 18 of \cite{ERZ}). In fact, this is a finite semisimple $\mathds{k}'$-algebra. Thus, by the Artin-Wedderburn theorem, $$End_{\widehat{\mathcal{M}}}(X) \cong \mathrm{Mat}_{n_1}(D_1)\times ...\times \mathrm{Mat}_{n_m}(D_m),$$ for some integers $m, n_1,..., n_m$, and finite division rings $D_1$, ..., $D_m$ over $\mathds{k}'$. In particular, for any simple summand $Y$ of $X$, we have $dim_{\mathds{k}}(End_{\widehat{\mathcal{M}}}(Y))\geq N+1$. As $X$ and $Y$ were arbitrary, this shows that $N(\mathcal{M}')\geq N+1$, and so the proof is complete.
\end{proof}

It turns out that we can describe those fields whose separable closure is a finite extension quite explicitly using a slight elaboration on the Artin-Schreier theorem. This is carried out in theorem 4.1 of \cite{C}, which we now recall.

\begin{Theorem}\label{thm:ArtinSchreier}
Let $\mathds{k}$ be a field whose separable closure is a finite extension. Then, either $\mathds{k}$ is separably closed, or $\mathds{k}$ is a real closed field, i.e. $char(\mathds{k})=0$, $\mathds{k}$ is not algebraically closed, and $\mathds{k}\lbrack\sqrt{-1}\rbrack$ is algebraically closed.
\end{Theorem}

\begin{Remark}
Examples of real closed fields include $\mathbb{R}$, but also $\mathbb{R}^{alg}=\overline{\mathbb{Q}}\cap\mathbb{R}$ the subfield of algebraic numbers. There are more exotic examples such as the fields of hyperreal numbers.
\end{Remark}

Finally, it is well-known that there is no nontrivial finite dimensional division algebra over an algebraically closed field. We need a weaker version of this statement for finite dimensional division algebras over real closed fields.

\begin{Lemma}\label{lem:divisionalgebrarealclosed}
Let $\mathds{k}$ be a real closed field. There are finitely many equivalences classes of finite dimensional division algebras over $\mathds{k}$.
\end{Lemma}
\begin{proof}
The proof of the Frobenius theorem on division algebras given in \cite{Pal} applies verbatim to any real closed field, showing that there are three equivalence classes of finite division rings occurring in this case: $\mathds{k}$, $\mathds{k}\lbrack\sqrt{-1}\rbrack$, and the ring of quaternions. This completes the proof.
\end{proof}

\subsection{Examples}

Let us fix a field $\mathds{k}$, which satisfies the condition of theorem \ref{thm:ArtinSchreier}. Further, let us assume that $\mathds{k}$ is perfect, as we are mainly interested in developing our theory over such fields. This means that $\mathds{k}$ is either algebraically closed or real closed, as a separably closed field is algebraically closed if and only if it is perfect.

\begin{Example}\label{ex:FpalgZp}
Let $\mathds{k} = \overline{\mathbb{F}}_p$ for a prime $p$. We fix $\mathcal{C}$ the category of $\mathbb{Z}/p$-graded finite-dimensional $\mathds{k}$-vector spaces. As this is a fusion 1-category, we find using proposition \ref{prop:modulesemisimple} that $\mathbf{Mod}(\mathcal{C})$ is a semisimple 2-category. Furthermore, by proposition 4.1 of \cite{EO} (see also \cite{Nat}), finite semisimple right $\mathcal{C}$-module 1-categories are classified by a subgroup $H\subseteq \mathbb{Z}/p$ together with a class in $H^2(H;\mathds{k}^{\times})$. However, those module 1-categories corresponding to $H=\mathbb{Z}/p$ are not separable. So we find that $\mathbf{Mod}(\mathcal{C})$ has only one equivalence class of simple objects, whence it is a finite semisimple 2-category, which is not locally separable.
\end{Example}

\begin{Theorem}\label{thm:modulefinitesemisimple}
Let $\mathcal{C}$ be a multifusion 1-category over an algebraically closed field or a real closed field. The 2-category of separable right $\mathcal{C}$-module 1-categories is a finite semisimple 2-category.
\end{Theorem}
\begin{proof}
By theorem \ref{thm:modulenearfinitesemisimple}, it is enough to prove that there are finitely many equivalence classes of indecomposable separable right $\mathcal{C}$-module 1-categories. We mimick the argument of corollary 9.1.6 of \cite{EGNO}.

By proposition 3.4.6 of \cite{EGNO}, there are only finitely many indecomposable $\mathbb{Z}_+$-modules over $Gr(\mathcal{C})$, the Grothendieck ring of $\mathcal{C}$. Moreover, as there are only finitely many equivalence classes of finite division algebras over $\mathds{k}$, there are only finitely many non-equivalent ways to categorify every such $\mathbb{Z}_+$-module. Therefore, it is enough to prove that for any finite semisimple 1-category $\mathcal{M}$, there are, up to equivalence, only finitely many separable right $\mathcal{C}$-module structures on $\mathcal{M}$. In order to do so, we will reduce this problem to a question about tensor functors, and appeal to theorem \ref{thm:nodeformation} below.

Supplying a right $\mathcal{C}$-module structure on $\mathcal{M}$ is equivalent to giving the data of a tensor functor $F:\mathcal{C}^{\otimes op}\rightarrow Fun(\mathcal{M},\mathcal{M})$. Observe that the right hand-side is a multifusion 1-category as $\mathds{k}$ is perfect. If the chosen right $\mathcal{C}$-module structure on $\mathcal{M}$ is separable, we claim that the tensor functor $F$ above exhibits $Fun(\mathcal{M},\mathcal{M})$ as a separable $\mathcal{C}^{\otimes op}\boxtimes\mathcal{C}$-module 1-category. Namely, proposition 2.4.10 of \cite{DSPS13} gives an equivalence $$Fun(\mathcal{M},\mathcal{M})\simeq Fun(\mathcal{M},\mathbf{Vect}_{\mathds{k}})\boxtimes \mathcal{M},$$ as $\mathcal{C}$-$\mathcal{C}$-bimodule 1-categories, which we can view as an equivalence of right $\mathcal{C}^{\otimes op}\boxtimes\mathcal{C}$-module 1-categories. Now, because the right $\mathcal{C}$-module 1-category $\mathcal{M}$ is separable, there exists a separable algebra $A$ in $\mathcal{C}$ such that $\mathcal{M}\simeq\mathrm{Mod}_{\mathcal{C}}(A)$ as right $\mathcal{C}$-module 1-categories. Combining proposition 2.4.9 and corollary 2.4.13 of \cite{DSPS13}, there is an equivalence of left $\mathcal{C}$-module 1-categories between $Fun(\mathcal{M},\mathbf{Vect}_{\mathds{k}})$ and the 1-category of right $A$-modules in $\mathcal{C}$. Equivalently, there is an equivalence of right $\mathcal{C}^{\otimes op}$-module 1-categories $Fun(\mathcal{M},\mathbf{Vect}_{\mathds{k}})\simeq \mathrm{Mod}_{\mathcal{C}^{\otimes op}}(A')$, where $A'$ is $A$ viewed as a separable algebra in $\mathcal{C}^{\otimes op}$. Finally, proposition 3.8 of \cite{DSPS14} shows that 

$$Fun(\mathcal{M},\mathbf{Vect}_{\mathds{k}})\boxtimes \mathcal{M}\simeq Mod_{\mathcal{C}^{\otimes op}\boxtimes\mathcal{C}}(A'\boxtimes A),$$ and it is straightforward to check that $A'\boxtimes A$ is a separable algebra in $\mathcal{C}^{\otimes op}\boxtimes\mathcal{C}$.

Consequently, it is enough to prove that there are only finitely many equivalence classes of tensor functors $F:\mathcal{C}^{\otimes op}\rightarrow Fun(\mathcal{M},\mathcal{M})$, which exhibits the right hand-side as a separable right $\mathcal{C}^{\otimes op}\boxtimes\mathcal{C}$-module 1-category. This is precisely the content of theorem \ref{thm:nodeformation}. This proves that there are finitely many equivalences classes of indecomposable separable right $\mathcal{C}$-module 1-categories.
\end{proof}

\begin{Corollary}\label{cor:separablefinite}
Over an algebraically closed field or a real closed field, a semisimple 2-category is finite if and only if it is compact.
\end{Corollary}

Over algebraically closed fields of positive characteristic, example \ref{ex:FpalgZp} demonstrates that local separability is an independent property. When $char(\mathds{k})=0$, more can be said.

\begin{Corollary}\label{cor:char0equivalence}
Over an algebraically closed field of characteristic zero or a real closed field, every compact semisimple 2-category is finite and locally separable.
\end{Corollary}
\begin{proof}
Over a field of characteristic zero, every multifusion 1-category is separable by corollary 2.6.8 of \cite{DSPS13}.
\end{proof}

\begin{Example}\label{ex:vectRC}
We work over $\mathbb{R}$. Lemma \ref{lem:divisionalgebrarealclosed} implies that the simple object of $\mathbf{Mod}(\mathbf{Vect}_{\mathbb{R}})$ are given by $\mathbf{Vect}_{\mathbb{R}}$, $\mathbf{Vect}_{\mathbb{C}}$, and $\mathbf{Vect}_{\mathbb{H}}$. But, note that $\mathbf{Mod}(\mathbf{Vect}_{\mathbb{C}})$ has only one simple object, $\mathbf{Vect}_{\mathbb{C}}$ itself.
\end{Example}

Let us conclude by examining an example of a compact semisimple 2-category over a separably closed but not algebraically closed field.

\begin{Example}
Let $\mathds{k} = \mathbb{F}_p(x)^{sep}$, the separable closure of the function field of $\mathbb{F}_p$. We claim that $\mathbf{Mod}(\mathbf{Vect}_{\mathds{k}})$ has only one equivalence class of simple object, $\mathbf{Vect}_{\mathds{k}}$ itself. Namely, as $\mathds{k}$ is separably closed, so is every finite extension. Further, the Brauer group of a separably closed field is trivial. Thus, the equivalence classes of indecomposable finite semisimple right $\mathbf{Vect}_{\mathds{k}}$-module 1-categories are given by the 1-categories $\mathbf{Vect}_{\mathds{k'}}$ for all finite extensions $\mathds{k}'$ of $\mathds{k}$. But, if $\mathds{k}'$ is a non-trivial extension, the right $\mathbf{Vect}_{\mathds{k}}$-module 1-category $\mathbf{Vect}_{\mathds{k'}}$ is clearly not separable.
\end{Example}

This leads us to ask the following question.

\begin{Question}
Over a separably closed field, is every compact semisimple 2-category finite?
\end{Question}

\subsection{Absence of Deformations for Tensor Functors}\label{sub:nodeformations}

Throughout, we work over $\mathds{k}$ a perfect field. Let $\mathcal{C}$, $\mathcal{D}$ be finite semisimple tensor 1-categories, and $F$ a tensor functor $\mathcal{C}\rightarrow \mathcal{D}$. The tensor functor $F$ endows $\mathcal{D}$ with the structure of a $\mathcal{C}$-$\mathcal{C}$-bimodule 1-category, or equivalently of a right $\mathcal{C}^{\otimes op}\boxtimes\mathcal{C}$-module 1-category. 

Following \cite{Da}, \cite{Ye98}, and \cite{Ye01}, we define $C^{\bullet}_{DY}(F)$, the Davydov-Yetter cochain complex of $F$ by setting for every integer $n\geq 0$: $$C^n_{DY}(F):= \textrm{Nat}(\otimes^{n}\circ F^{\times n},\otimes^{n}\circ F^{\times n});$$ \begin{align*}\delta(f)_{X_0,...,X_n}:=& Id_{F(X_0)}\otimes f_{X_1,...,X_n}\\ & +\sum_{i=1}^n (-1)^i f_{X_0,...,X_{i-1}\otimes X_i,...,X_n}\\&+(-1)^{n+1}f_{X_0,...,X_{n-1}}\otimes Id_{F(X_n)}.\end{align*}
The cohomology groups of the cochain complex $C^{\bullet}_{DY}(F)$ are denoted by $H^{\bullet}_{DY}(F)$, and called the Davydov-Yetter cohomology groups of $F$. These cohomology groups describe the deformations of the tensor functor $F$.

\begin{Proposition}\label{prop:vanishingDY}
Provided that $\mathcal{D}$ is separable when viewed as a right $\mathcal{C}^{\otimes op}\boxtimes\mathcal{C}$-module 1-category, the group $H^n_{DY}(F)$ is trivial for every integer $n\geq 1$.
\end{Proposition}
\begin{proof}
We claim that the argument of corollary 3.18 of \cite{GHS} applies verbatim in our more general context. Namely, the only thing we have to check is that $Z_F$-$\mathrm{Mod}$ is finite semisimple (see subsection 3.3 of \cite{GHS} for the definition). In order to do so, observe that $Z_F$-$\mathrm{Mod}\simeq \mathcal{Z}_{\mathcal{C}}(\mathcal{D})$ the relative Drinfel'd center by proposition 3.10 of \cite{GHS}. Now, we have assumed that $\mathcal{D}$ is separable as a $\mathcal{C}^{\otimes op}\boxtimes\mathcal{C}$-module 1-category, thus $\mathcal{Z}_{\mathcal{C}}(\mathcal{D})\simeq Fun_{\mathcal{C}^{\otimes op}\boxtimes\mathcal{C}}(\mathcal{C},\mathcal{D})$ is a finite semisimple 1-category, which concludes the proof.
\end{proof}

\begin{Remark}
The condition of proposition \ref{prop:vanishingDY} is always satisfied if $\mathcal{C}$ is separable. Namely, under this hypothesis, $\mathcal{C}^{\otimes op}\boxtimes \mathcal{C}$ is separable by corollary 2.5.11 of \cite{DSPS13}. Thus, as $\mathcal{D}$ is finite semisimple, it it separable as a right $\mathcal{C}^{\otimes op}\boxtimes \mathcal{C}$-module 1-category.
\end{Remark}

Our goal is now to generalize theorem 2.31 of \cite{ENO} (see also theorem 9.1.5 of \cite{EGNO}), which holds over an algebraically closed field of characteristic zero. We wish to point out that it is asserted in section 9 of \cite{ENO} that their theorem 2.31 holds over an arbitrary algebraically closed field, so long as $\mathcal{C}$ is separable. When the base field is algebraically closed, our proof is inspired by that given in section 9 of \cite{EGNO} in which they work over an algebraically closed field of characteristic zero, though we give a detailed explanation as to why certain orbits of an action of an algebraic group are open. When the base field is real closed, the setup of classical algebraic geometry is no longer adequate, and we need to use ideas coming from real algebraic geometry (see \cite{BCR}).

\begin{Theorem}\label{thm:nodeformation}
Let $\mathds{k}$ be an algebraically closed, or a real closed field. A tensor functor $F:\mathcal{C}\rightarrow \mathcal{D}$ satisfying the hypothesis of proposition \ref{prop:vanishingDY} does not admit any non-trivial deformation. In particular, the number of such functors up to equivalence (for fixed source and target) is finite.
\end{Theorem}
\begin{proof}
\textbf{Case 1: $\mathds{k}$ is algebraically closed}

\textit{Defining the algebraic group action:} We may assume that $\mathcal{C}$ and $\mathcal{D}$ are skeletal. We write $C_i$, $i=0,...,n$ for the set of simple objects of $\mathcal{C}$. We think of the assignment $C_i\mapsto D_i:=F(C_i)$ as fixed. As $\mathds{k}$ is algebraically closed, this is exactly the data of a linear functor $F':\mathcal{C}\rightarrow \mathcal{D}$. A choice of tensor structure for the functor $F':\mathcal{C}\rightarrow \mathcal{D}$ corresponds to the choice of isomorphism $\phi_{C_i,C_j}:F'(C_i)\otimes F'(C_j)\rightarrow F'(C_i\otimes C_j)$ in finite dimensional $\mathds{k}$-vector space for every $i,j$ satisfying finitely many polynomial conditions. Thus, there is a (potentially non-irreducible) variety $X$ over $\mathds{k}$ whose closed points are exactly the possible choices of tensor structure on $F'$.

Similarly, we define $G$ to be the algebraic group over $\mathds{k}$ of natural isomorphisms $\lambda:F\cong F$. As the underlying scheme of $G$ is again a (potentially non-irreducible) variety, it is reduced, which implies that $G$ is smooth (see  proposition 1.18 of \cite{Mil}). The algebraic group $G$ acts on the variety $X$ by sending $(\lambda, \{\phi_{C_i,C_j}\}_{i,j})$ in $G\times X$ to $\{\lambda_{C_i\otimes C_j}\phi_{C_i,C_j}(\lambda_{C_i}^{-1}\otimes \lambda_{C_j}^{-1})\}_{i,j}$ in $X$.

\textit{Openness of orbits:} Fix $x$ a closed point of $X$, and write $F_x$ for the corresponding tensor functor. We wish to consider the action map $l_x:G\rightarrow X$ given by sending $g$ to $g\cdot x$. The differential of $l_x$ at $e\in G$ is a map $(dl_x)_e: T_eG\rightarrow T_xX$. It follows from the definition of the Davydov-Yetter chain complex that $T_xX = Z^2_{DY}(F_x)$, the set of 2-cochains, and that $(dl_x)_e (T_eG) = B^2_{DY}(F_x)$, the set of 2-coboundaries. Now, if $F_x$ exhibits $\mathcal{D}$ as a separable right $\mathcal{C}^{\otimes op}\boxtimes\mathcal{C}$-module 1-category, $H^2_{DY}(F_x)=Z^2_{DY}(F_x)/B^2_{DY}(F_x)$ is trivial by proposition \ref{prop:vanishingDY}, so $(dl_x)_e$ is a surjection.

We now wish to prove that every $x$ in $X$ such that $(dl_x)_e$ is a surjection is a smooth point. Let $x$ in $X$ be such a point. We claim that there exists a subvariety $Y$ in $X$ containing $x$, which is stable under the action of $G$, and for which $x$ is a smooth point. Namely, if $x$ is not a smooth point of $X$, it belongs to the singular locus $Y_1$ of $X$, which is stable under the action of $G$. If $x$ is a smooth point of $Y_1$, we are done. If note, $x$ belongs to the singular locus of $Y_1$, etc. This process eventually stops because the singular locus of an irreducible variety has strictly smaller dimension than the original variety. Now, let $Y$ be the desired variety. The inclusion $Y\hookrightarrow X$ induces a commutative diagram $$\begin{tikzcd}
T_eG \arrow[rr, "(dl_y)_e", two heads] \arrow[rd] &                       & T_xX. \\
                                                  & T_xY \arrow[ru, hook] &     
\end{tikzcd}$$ In particular, the map $T_xY\rightarrow T_xX$ is an isomorphism. This implies that the inclusion induces an isomorphism $\mathcal{O}_{Y,x}\cong \mathcal{O}_{X,x}$ of local rings (the first ring is regular), and so $X$ and $Y$ agree locally around $x$, which proves that $x$ is a smooth point of $X$.

We can now put our discussion together to prove the first part of the result. Namely, fix $x$ a closed point of $X$ such that $(dl_x)_e$ is surjective. As $G$ is an algebraic group, this implies that for every closed point $g$ of $G$, the map $(dl_x)_g:T_gG\rightarrow T_{g\cdot x}X$ is surjective. Thanks to the discussion above, this implies that $l_x:G\rightarrow X$ lands in a smooth subvariety of $X$. But, by proposition 10.4 of \cite{Har}, this implies that $l_x$ is smooth, and by corollary 14.34 of \cite{GoWe}, we get that $l_x$ is open. Thus, we find that the set-theoretic image $l_x(G)$, i.e the orbit $G\cdot x$, is open in $X$. In particular, this applies to those closed points $x$ for which the tensor functor $F_x:\mathcal{C}\rightarrow \mathcal{D}$ exhibits $\mathcal{D}$ as a separable right $\mathcal{C}^{\otimes op}\boxtimes\mathcal{C}$-module 1-category, in which case this shows that $F_x$ admits no non-trivial deformations. In particular, this applies to the tensor functor $F$ itself.

\textit{Conclusion:} Let us write $X'$ for the smooth subscheme of $X$ given by the union of the orbit under $G$ of the closed points $x$ of $X$ for which $(dl_x)_e$ is surjective. By definition, we can write $$X' = \bigcup_{x\in X'}G\cdot x.$$ But, as $X'$ is an open subset of a variety, it is quasicompact, so $X'$ is the union of finitely many orbit. The second part of the theorem follows from this last fact by combining it with lemma \ref{lem:Z+homfinite}. Namely, this lemma asserts that there are finitely many homomorphisms of $\mathbb{Z}_+$-rings from $Gr(\mathcal{C})$ to $Gr(\mathcal{D})$, i.e. there are only finitely many assignments $C_i\mapsto D_i$, as we have fixed at the start of the proof, that can be made into tensor functors.

\vspace{1ex}
\noindent \textbf{Case 2: $\mathds{k}$ is a real closed field}

\textit{Deformations of linear functors:}  We may assume that $\mathcal{C}$ and $\mathcal{D}$ are skeletal. We write $C_i$, $i=0,...,n$ for the set of simple objects of $\mathcal{C}$. We fix an assignment $C_i\mapsto D_i$ an objects of $\mathcal{D}$. Now, $\mathds{k}$ is not algebraically closed so this does not determine a linear functor $\mathcal{C}\rightarrow\mathcal{D}$ uniquely. But, we claim that there are only finitely many equivalence classes of such functors. Namely, upgrading the above assignment to a functor amounts to choosing a homomorphism of $\mathds{k}$-algebras $f_i:End(C_i)\rightarrow End(F(C_i))$ for all $i$. Note that it is enough to prove that up to conjugation by an invertible element of $End(F(C_i))$, there are only finitely many such homomorphisms. Namely, if $f_i$ and $f_i'$ are two choices of homomorphisms such that for every $i$, we have $f_i = a_if_i'a_i^{-1}$ for some invertible $a_i$ in $End(F(C_i))$, then the $a_i$ define a natural isomorphism between the functor defined using the homomorphisms $f_i$ and that defined using the homomorphisms $f_i'$.

So it is enough to prove that for every finite dimensional division algebra $A$ over $\mathds{k}$, and finite semisimple $\mathds{k}$-algebra $B$, there are finitely many $\mathds{k}$-algebra homomorphisms $A\rightarrow B$ up to conjugation by invertible elements of $B$. Without loss of generality, we may assume that $B$ is simple, i.e. $B$ is the algebra of $m\times m$-matrices on $\mathds{k}$, $\mathds{k}\lbrack\sqrt{-1} \rbrack$, or the quaternions. Now, if $B$ is the algebra of $m\times m$-matrices on $\mathds{k}$, or the quaternions, the claim follows from the Skolem-Noether theorem (see theorem 2.7.2 of \cite{GS}). Thus, it only remains to check the claim when $B$ is the algebra of $m\times m$-matrices on $\mathds{k}\lbrack\sqrt{-1} \rbrack$. If $A=\mathds{k}$, the statement is trivial. If $A=\mathds{k}\lbrack\sqrt{-1}\rbrack$, then a $\mathds{k}$-algebra homomorphism $f:A\rightarrow B$ is uniquely determined by the image of $\sqrt{-1}$. But $f(\sqrt{-1})^{2}= Id$, so $f(\sqrt{-1})$ is conjugate to a diagonal matrix with entries $\pm\sqrt{-1}$. There are clearly finitely many of those up to conjugation. Finally, if $A$ is the quaternions, let $f':A\otimes \mathds{k}\lbrack\sqrt{-1} \rbrack\rightarrow B$ be the extension of $f$ to a homomorphism of $\mathds{k}\lbrack\sqrt{-1} \rbrack$-algebras. It is enough to prove that there are finitely many $\mathds{k}\lbrack\sqrt{-1} \rbrack$-algebra homomorphisms $A\otimes \mathds{k}\lbrack\sqrt{-1} \rbrack\rightarrow B$ up to conjugation by invertible elements of $B$. But $A\otimes \mathds{k}\lbrack\sqrt{-1}\rbrack\cong Mat_2(\mathds{k}\lbrack\sqrt{-1}\rbrack)$ is simple, so the claim follows again from the Skolem-Noether theorem.

\textit{Defining the algebraic group action:} Let us now fix a $\mathds{k}$-linear functor $F':\mathcal{C}\rightarrow \mathcal{D}$. A choice of tensor structure on $F'$ amounts to the choice of an isomorphism $\phi_{C_i,C_{j}}:F'(X_i)\otimes F'(X_{j})\rightarrow F'(X_i\otimes X_{j})$ in a finite dimensional $\mathds{k}$-vector space for every $i,j$. Furthermore, these isomorphisms have to satisfy certain polynomial equations guaranteeing that $\phi$ is a natural transformation and suitably coherent. Thus, there is a (potentially non-irreducible) Zariski-open subset $X$ of an algebraic set in $\mathds{k}^{\times N}$ for some positive integer $N$, whose points are exactly the possible tensor structures on $F$. Similarly, we define $G$ to be the Zariski-open subset $\mathds{k}^{\times M}$ with $M$ a positive integer, whose points are the natural isomorphisms $\lambda:F\cong F$. Observe that the Zariski-open subset $G$ carries a natural continuous action given by composition. Moreover, $G$ acts on $X$ by sending $(\lambda, \{\phi_{C_i,C_j}\}_{i,j})$ in $G\times X$ to $\{\lambda_{C_i\otimes C_j}\phi_{C_i,C_j}(\lambda_{C_i}^{-1}\otimes \lambda_{C_j}^{-1})\}_{i,j}$ in $X$. This action is continuous in the Zariski topology.

\textit{Openness of orbits:} We claim that $X$ is smooth. To see this, fix a closed point $x\in X$, and write $F_x$ for the corresponding tensor functor. We consider the action map $l_x:G\rightarrow X$ given by sending $g$ to $g\cdot x$. Its differential at $e\in G$ is a map $(dl_x)_e: T_eG\rightarrow T_xX$ between the Zariski tangent spaces. Now, it follows from the definition of the Davydov-Yetter chain complex that $T_xX = Z^2_{DY}(F_x)$ the set of 2-cochains, and that $(dl_x)_e (T_eG) = B^2_{DY}(F_x)$ the set of 2-coboundaries. By corollary 2.6.9 of \cite{DSPS13}, the separability assumption of proposition \ref{prop:vanishingDY} is satisfied, so that $H^2_{DY}(F_x)=Z^2_{DY}(F_x)/B^2_{DY}(F_x)$ is trivial, and $(dl_x)_e$ is surjective.

Using an inductive argument similar to the one we have used in the previous case together with proposition 3.3.14 of \cite{BCR}, we can construct an algebraic subset $Y$ of $X$ containing $x$ as a smooth point, and stable under the action of $G$. This implies that $l_x$ factors through $Y$, so that $(dl_x)_e$ factors through $T_xY$. But $(dl_x)_e$ is surjective, which implies that the canonical inclusion $T_xY\hookrightarrow T_xX$ is an isomorphism. This shows that $Y$ is a Zariski-open subset of $X$ containing $x$, and proves that $X$ is smooth at $x$.

In order to finish our argument, we need to consider a finer topology than the Zariski topology. Namely, as $\mathds{k}$ is a real closed field, it is in particular well-ordered by theorem 1.2.2 of \cite{BCR}. This allows us to consider the Euclidean topology on the spaces $\mathds{k}^{\times N}$ and $\mathds{k}^{\times M}$, which is strictly finer than the Zarisky topology. For the remainder of the proof, we will exclusively work with the Euclidean topology. In particular, we think of the algebraic sets $X$, and $G$ as being endowed with the Euclidean topology, i.e. we think of $X$, and $G$ as semi-algebraic sets in the terminology of \cite{BCR}.

Let us fix a point $x$ in $X$, we will now prove that the orbit of $x$ in $X$ under the action of $G$ is open in the Euclidean topology. As $x$ is a smooth point of an algebraic set, proposition 3.3.11 of \cite{BCR} shows that there exists an open subset $U$ of $x$ in $X$, an open subset $V$ of $\mathds{k}^{\times P}$, and a Nash diffeomorphism $\phi:U\cong V$. (A Nash function between semi-algebraic sets is a smooth semi-algebraic map.) Let us now consider the Nash function $\phi\circ l_x:l_x^{-1}(U)\rightarrow V$, which maps $g$ in $l_x^{-1}(U)\subseteq G$ to $\phi(g\cdot x)$. But $(dl_x)_e$ is surjective, thence it follows from proposition 2.9.7 of \cite{BCR} that the image of $\phi\circ l_x$ contains an open neighborhood of $\phi(x)$. Given that $\phi$ is a Nash diffeomorphism, this implies that the image of $l_x:G\rightarrow X$ contains an open neighborhood of $x$. As $x$ was arbitrary we find that $G\cdot x$ is an open subset of $X$, i.e. the associated tensor functor $F_x$ admits no non-trivial deformations.

\textit{Conclusion:} Proposition 2.2.7 of \cite{BCR} shows that for any $x$ in $X$, $G\cdot x$ is semi-agebraic. As the orbits in $X$ under the action of $G$ are disjoint, we therefore find that $G\cdot x$ is an open and closed semi-algebraic subset of $X$, i.e. a union of semi-algebraically connected components of $X$. But, $$X = \bigcup_{x\in X}G\cdot x,$$ and $X$ is a finite union of semi-algebraically connected components by theorem 2.4.4 of \cite{BCR}. Thus, $X$ is the union of finitely many orbits under the action of $G$. The second part of the statement now follows from lemma \ref{lem:Z+homfinite}.
\end{proof}

\begin{Example}
Let $\mathds{k} = \mathbb{Q}$, and write $\mathcal{C}$ for the separable finite semisimple tensor 1-category of finite dimensional $\mathbb{Z}/2$-graded $\mathbb{Q}$-vector spaces. Then, there are infinitely many non-equivalent tensor functors $\mathcal{C}\rightarrow \mathbf{Vect}_{\mathds{k}}$. Namely, such functors up to equivalence correspond exactly to classes in $H^2(\mathbb{Z}/2;\mathbb{Q}^{\times})\cong \mathbb{Q}^{\times}/(\mathbb{Q}^{\times})^2$, and this group is infinite. In the notation of the proof of theorem \ref{thm:nodeformation}, $X=\mathbb{Q}^{\times}$, $G=\mathbb{Q}^{\times}$, and $G$ acts on $X$ by $(g,x)\mapsto g^2x$. If we view $X$ and $G$ as sets, the orbits of this action are manifestly not open. However, if we view them as schemes, it is possible to show that the orbits are indeed open. The reason for this discrepancy is that schemes over $\mathbb{Q}$ have many closed points which are not $\mathbb{Q}$-points.
\end{Example}

\subsection{Absence of Deformations for Separable Tensor 1-Categories}

Let us assume that $\mathds{k}$ is an algebraically closed field or a real closed field. We begin by briefly explaining how the methods of the previous section can be used to generalize theorem 2.30 of \cite{ENO}, which holds over an algebraically closed field of characteristic zero. Further, over an algebraically closed field of positive characteristic, this statement appears in section 9 of \cite{ENO}, and a proof can be obtained using the techniques developed in their article. We sketch a proof in the algebraically closed case both for completeness, and in order to introduce some notations that will be used in what follows. We emphasize that the real closed case has not been previously considered in the literature.

\begin{Theorem}\label{thm:nodeformationmulti}
Let $\mathcal{C}$ be a separable multifusion 1-category over an algebraically closed or real closed field. Then, $\mathcal{C}$ has no non-trivial deformations. Consequently, the number of equivalence classes of separable multifusion 1-categories with a fixed Grothendieck ring is finite.
\end{Theorem}

\begin{proof}
The proof is entirely analogous to that of theorem \ref{thm:nodeformation}, so that we only indicate the key steps.

\vspace{1ex}
\noindent \textbf{Case 1: $\mathds{k}$ is algebraically closed}

We may assume that $\mathcal{C}$ is skeletal. We write $C_i$, $i=0,...,n$ for its simple objects, and assume that $C_0 =I$ is the monoidal unit. The tensor structure of $\mathcal{C}$ establishes for us an assignment $(C_i,C_j)\mapsto \oplus_k n^k_{ij}C_k$ for some non-negative integers $n^k_{ij}$, which we think as fixed. As $\mathcal{C}$ is rigid and skeletal, for any $i$, the left adjoint $C_i^*$ is equal to $C_{i^*}$, for some integer $i^*$ in $\{0,...,n\}$. We fix left and right coevalutation morphisms $coev_{i}:I\rightarrow C_i\otimes C_{i^*}$ and $coev'_{i}:I\rightarrow C_{i^*}\otimes C_i$ for every $i$ in $\mathcal{C}$. Let us now define $Z$ to be the (potentially non-irreducible) variety over $\mathds{k}$ whose closed points are precisely the data necessary to turn the above assignment into a tensor 1-category, i.e. the coherent choices of associators, and evaluation morphisms compatible with the associators and the fixed coevaluation morphisms. Similarly, we define a smooth algebraic group $H$ whose closed points are the automorphisms of the identity functor on $\mathcal{C}$ (viewed as a linear 1-category). The group structure is given by composition, and $H$ acts on $Z$.

Given a point $z$ in $Z$, write $\mathcal{C}_z$ for the corresponding multifusion 1-category. We consider the action map $l_z:H\rightarrow Z$ given by $h\mapsto h\cdot z$. The differential of $l_z$ at $e$ is a map $(dl_z)_e:T_eG\rightarrow T_zX$. If $\mathcal{C}_z$ is separable, the Davydov-Yetter cohomology group $H^3_{DY}(\mathcal{C}_z)$ vanishes. This can be seen by following the original argument given in \cite{ENO} carefully. More straightforwardly, it is also also possible use the fact that the proof of theorem 9.1.3 of \cite{EGNO} applies to any separable fusion 1-category over an algebraically closed field. (One may alternatively appeal to proposition \ref{prop:vanishingDY} above.) As a consequence, we find that $(dl_z)_e$ is surjective. As in the proof of theorem \ref{thm:nodeformation}, any point $z$ for which $(dl_z)_e$ is surjective is smooth. Moreover, the orbit $H\cdot z$ of any such point $z$ is open, so that $\mathcal{C}_z$ admits no nontrivial deformations. This applies in particular to $\mathcal{C}$.

Finally, write $Z'$ for the smooth subvariety of $Z$ given by the union of the orbits under $H$ of the closed points $z$ in $Z$ such that $(dl_z)_e$ is surjective. Then, we have $$Z'=\bigcup_{z\in Z'}H\cdot z.$$ But, $Z'$ being an open subset of a variety, it is quasi-compact, which shows that $Z'$ is the union of finitely many orbits. This proves the second part of the statement.

\vspace{1ex}
\noindent \textbf{Case 2: $\mathds{k}$ is a real closed field}

Without loss of generality, we may assume that $\mathcal{C}$ is skeletal. Analogously to the algebraically closed case, we define a (potentially non-irreducible) Zariski open subset $Z$ of an algebraic set in $\mathds{k}^{\times N}$ for some positive integer $N$, whose points parameterise the tensor structure on $\mathcal{C}$. We also define $H$ to be the Zariksi-open subset of $\mathds{k}^{\times M}$ with $M$ a positive integer, whose points are the automorphisms of the identity functor on the linear 1-category $\mathcal{C}$. Observe that $H$ acts on $Z$ in a canonical way.

Given a point $z$ in $Z$, write $\mathcal{C}_z$ for the corresponding multifusion 1-category. We consider the action map $l_z:H\rightarrow Z$ given by $h\mapsto h\cdot z$. The differential of $l_z$ at $e$, the identity of $H$, is a map $(dl_z)_e:T_eG\rightarrow T_zX$. Thanks to proposition \ref{prop:vanishingDY} and the fact that every multifusion 1-category over $\mathds{k}$ is separable, the Davydov-Yetter cohomology group $H^3_{DY}(\mathcal{C}_z)$ vanishes. As a consequence, we find that $(dl_z)_e$ is surjective for every $z$, and the argument used in the proof of theorem \ref{thm:nodeformation} then shows that $Z$ is smooth. Further, by passing to the Euclidean topology, which is possible as $\mathds{k}$ is a real closed field, one shows that the orbit $H\cdot z$ of any point $z$ in $Z$ is an open semi-algebraic subset. This implies that $\mathcal{C}_z$ admits no nontrivial deformations.

Finally, for any $z$ in $Z$, the subset $H\cdot z$ is open and closed, meaning that it is a union of semi-algebraically connected components of $Z$. Further, we have $$Z=\bigcup_{z\in Z}H\cdot z.$$ But, by theorem 2.4.4 of \cite{BCR}, $Z$ is a finite union of semi-algebraically connected components. This implies that $Z$ is the union of finitely many orbits under the action of $H$, thereby establishing the second part of the theorem.
\end{proof}

\begin{Remark}
Over a general field, the statement of theorem \ref{thm:nodeformation} does not hold. Namely, working over $\mathbb{F}_p$, $\mathbf{Vect}_{\mathbb{F}_{p^q}}$ is a finite semisimple tensor 1-category for every positive integer $q$. These finite semisimple tensor 1-categories are manifestly separable and non-equivalent. Furthermore, they all share the same Grothendieck ring, $\mathbb{Z}$.
\end{Remark}

Over an algebraically closed field, corollary \ref{cor:nodefHopf} first appeared in \cite{Ste}, and corollary \ref{cor:nodefbraiding} follows from remark 2.33 of \cite{ENO}. Over a real closed field, both corollaries are new.

\begin{Corollary}\label{cor:nodefHopf}
There are only finitely many equivalence classes of semisimple and cosemisimple Hopf algebras of a fixed dimension $d$.
\end{Corollary}
\begin{proof}
Recall that, for any Hopf algebra $H$ of dimension $d$, the Frobenius-Perron dimension of the fusion 1-category $\mathbf{Rep}(H)$ is $d$. Further, from the reconstruction theorem, it follows that there is a bijection between equivalence classes of semisimple Hopf algebras of dimension $d$ and fusion 1-categories equipped with a tensor functor to $\mathbf{Vect}_{\mathds{k}}$. But, for a semisimple Hopf algebra $H$, it is argued in section 9.1 of \cite{EO}, that $\mathbf{Rep}(H)$ is separable if and only $H$ is cosemisimple. The proof is finished by appealing to theorems \ref{thm:nodeformation} and \ref{thm:nodeformationmulti}.
\end{proof}

\begin{Corollary}\label{cor:nodefbraiding}
There are only finitely many equivalence classes of braidings on a given separable fusion 1-category.
\end{Corollary}
\begin{proof}
A braiding on a fusion 1-category $\mathcal{C}$ can equivalently be described as an tensor functor $\mathcal{C}\rightarrow \mathcal{Z}(\mathcal{C})$ whose composite with the forgetful tensor functor $\mathcal{Z}(\mathcal{C})\rightarrow \mathcal{C}$ is the identity. Now, if $\mathcal{C}$ is separable, $ \mathcal{Z}(\mathcal{C})$ is a finite semisimple 1-category. Furthermore, $ \mathcal{Z}(\mathcal{C})$ is a separable fusion 1-category. Over an algebraically closed field, this follows from theorem 2.15 of \cite{ENO}, and, over a real closed field, every fusion 1-category is separable. In particular, theorem \ref{thm:nodeformation} applies, which proves the result.
\end{proof}

Finally, let us recall corollary 9.1.8 of \cite{EGNO} as corollary \ref{cor:definedoverQbar} below, and give the appropriate generalizations in corollaries \ref{cor:definedoverFpbar} and \ref{cor:definedoverRpalg}. Let us also mention that corollary \ref{cor:definedoverFpbar} is certainly known to experts, as it is mentioned in subsection 4.5 of \cite{Et}.

\begin{Corollary}\label{cor:definedoverQbar}
Let $\mathds{k}$ be an algebraically closed field of characteristic zero, then any multifusion 1-category, tensor functor between them, finite semisimple module 1-category over them, and semisimple Hopf algebra is defined over $\overline{\mathbb{Q}}$.
\end{Corollary}
\begin{proof}
Clearly, $\overline{\mathbb{Q}}$ is a subfield of $\mathds{k}$. Now, observe that the variety $Z$ over $\mathds{k}$ used in the proof of theorem \ref{thm:nodeformationmulti} is defined over $\mathbb{Z}$. Let us write $P(Z)$ for the corresponding algebraic set in $\mathds{k}^{\times m}$, which is also defined over $\mathbb{Z}$. Then, using lemma \ref{lem:density} below, we see that the set of $\overline{\mathbb{Q}}$-points of $P(Z)$ is dense in the Zariski topology. Let $\mathcal{C}$ be a multifusion 1-category corresponding to a point $z\in P(Z)$ (i.e. a closed point of $Z$). As the orbit $H\cdot z$ is open in $P(Z)$, there exists a point $z'$ in $H\cdot z$ whose coordinates are in $\overline{\mathbb{Q}}$. This proves the claim for multifusion 1-categories.

In order to establish the claim for tensor functors, one can use the result we have just proven, and therefore assume that our multifusion 1-categories are defined over $\overline{\mathbb{Q}}$. This implies that the variety $X$ over $\mathds{k}$ used in the proof of theorem \ref{thm:nodeformation} is defined over $\overline{\mathbb{Q}}$. Then, we can apply the same reasoning as above to deduce the desired result.

Let $\mathcal{C}$ be a multifusion 1-category, which we can assume to be defined over $\overline{\mathbb{Q}}$. Note that for any finite semisimple 1-category $\mathcal{M}$, $End(\mathcal{M})$ is canonically defined over $\overline{\mathbb{Q}}$. Now, observe that a finite semisimple right $\mathcal{C}$-module 1-category $\mathcal{M}$ is defined over $\overline{\mathbb{Q}}$ if and only if the corresponding tensor functor $\mathcal{C}^{\otimes op}\rightarrow End(\mathcal{M})$ is defined over $\overline{\mathbb{Q}}$. Thus, the result for finite semisimple module 1-categories follows from that for tensor functors.

Finally, let $H$ be a semisimple Hopf algebra. By what we have prove above, there exists a fusion 1-category $\mathcal{C}$ defined over $\overline{\mathbb{Q}}$, which is equivalent to $\mathbf{Rep}(H)$. Furthermore, there is a tensor functor $F:\mathcal{C}\rightarrow\mathbf{Vect}$ defined over $\overline{\mathbb{Q}}$, which is equivalent to the canonical tensor functor $\mathcal{C}\simeq \mathbf{Rep}(H)\rightarrow\mathbf{Vect}$. Applying the reconstruction theorem to $F$ yields a Hopf algebra $H'$, which is defined over $\overline{\mathbb{Q}}$, and equivalent to $H$.
\end{proof}

The following lemma is certainly well-known. We include a proof for completeness.

\begin{Lemma}\label{lem:density}
Let $\mathds{k}\subseteq \mathds{K}$ be algebraically closed fields. Let $f_1,...,f_m$ be polynomials in $\mathds{k}\lbrack x_1,...,x_n\rbrack$, and write $V_{\mathds{k}}$, $V_{\mathds{K}}$ respectively, for the vanishing sets of the $f_i$'s in $\mathds{k}^{\times n}$, $\mathds{K}^{\times n}$ respectively. Then, $V_{\mathds{k}}$ is dense in $V_{\mathds{K}}$ with respect to the Zariski topology.
\end{Lemma}
\begin{proof}
We have to show that if a polynomial in $\mathds{K}\lbrack x_1,...,x_n\rbrack$ vanishes on $V_{\mathds{k}}$, then it vanishes on $V_{\mathds{K}}$. We prove the converse, i.e. that any polynomial $g$ in $\mathds{K}\lbrack x_1,...,x_n\rbrack$ which does not vanish $V_{\mathds{K}}$, does not vanish on $V_{\mathds{k}}$. Let $(p_1,...,p_n)$ be a point in $V_{\mathds{K}}$ for which $g(p_1,...,p_n)\neq 0$. Let us consider $A$ the $\mathds{k}$-subalgebra of $\mathds{K}$ generated by the coefficients of $g$, the $p_i$'s, and $g(p_1,...,p_n)^{-1}$. Let $\mathfrak{m}$ be a maximal ideal of $A$. As $A$ is finitely generated and $\mathds{k}$ is algebraically closed, we have $A/\mathfrak{m}\cong \mathds{k}$. Let $q_1,...,q_n$ be the images of $p_1,..,p_n$ along the quotient $A\twoheadrightarrow A/\mathfrak{m}\cong \mathds{k}$. We have $g(q_1,...,q_n) = g(p_1,...,p_n)$ in $A/\mathfrak{m}$, which is a unit by construction.
\end{proof}

\begin{Corollary}\label{cor:definedoverFpbar}
Let $\mathds{k}$ be an algebraically closed field of characteristic $p$, any separable multifusion 1-category, separable module 1-category over a separable multifusion 1-category, tensor functor between separable multifusion 1-categories, semisimple and cosemisimple Hopf algebra is defined over $\overline{\mathbb{F}}_p$.
\end{Corollary}
\begin{proof}
With the obvious modifications, one can use the argument used to prove corollary \ref{cor:definedoverQbar}.
\end{proof}

\begin{Corollary}\label{cor:definedoverRpalg}
Let $\mathds{k}$ be a real closed field, then any multifusion 1-category, tensor functor between them, finite semisimple module 1-category over them, and semisimple Hopf algebra is defined over $\mathbb{R}^{alg}=\overline{\mathbb{Q}}\cap \mathbb{R}$.
\end{Corollary}
\begin{proof}
Modifying the argument used to prove corollary \ref{cor:definedoverQbar} by replacing $\overline{\mathbb{Q}}$ by $\mathbb{R}^{alg}$ and by appealing to proposition 5.3.5 of \cite{BCR} instead of lemma \ref{lem:density} yields the desired result.
\end{proof}

\section{Additional Properties of Compact Semisimple 2-Categories}\label{sec:additional}

We explain how a number of important results in the theory of finite semisimple 2-categories generalize to compact semisimple 2-categories over an arbitrary field $\mathds{k}$.

\subsection{An Equivalence of 3-categories}

In the spirit of \cite{D1}, we establish an equivalence between two 3-categories. In order to introduce the first one, recall from \cite{DSPS13} following \cite{JS} that, over any base field, there is a 3-category $\mathbf{TC}$ of finite tensor 1-categories, finite bimodule 1-categories, bimodule functors, and bimodule natural transformations.

\begin{Theorem}
Finite semisimple tensor 1-categories, right separable bimodule 1-categories, bimodule functors, and bimodule natural transformations define a sub-3-category $\mathbf{TC}^{rsep}$ of $\mathbf{TC}$.
\end{Theorem}
\begin{proof}
It is enough to prove that right separable bimodule 1-categories compose. This is precisely what the proof of theorem 2.5.5 of \cite{DSPS13} shows.
\end{proof}

\begin{Remark}
Let $\mathds{k}$ be a perfect field, the full sub-3-category of $\mathbf{TC}^{rsep}$ on the separable finite semisimple tensor 1-categories is denoted by $\mathbf{TC}^{sep}$ in \cite{DSPS13}. Namely, proposition 2.5.10 of \cite{DSPS13} shows that $\mathbf{TC}^{sep}$ can also be described as the 3-category of separable finite semisimple tensor 1-categories, right separable bimodule 1-categories, bimodule functors, and bimodule natural transformations.
\end{Remark}

\begin{Definition}
We write $\mathbf{CSS2C}$ for the 3-category of compact semisimple 2-categories, linear 2-functors, 2-natural transformations, and modifications.
\end{Definition}

\begin{Theorem}\label{thm:equivalence}
Over an arbitrary field $\mathds{k}$, there is an equivalence of linear 3-categories $$\mathbf{Mod}:\mathbf{TC}^{rsep}\xrightarrow{\simeq} \mathbf{CSS2C}.$$
\end{Theorem}
\begin{proof}
The proof follows that of theorem 2.2.2 of \cite{D1}. As $\mathbf{Vect}_{\mathds{k}}$ is a finite semisimple tensor 1-category, we have a linear 3-functor $$Hom_{\mathbf{TC}^{rsep}}(\mathbf{Vect}_{\mathds{k}},-):\mathbf{TC}^{rsep}\rightarrow \mathbf{Cat}^2_{\mathds{k}},$$ to the 3-category of $\mathds{k}$-linear 2-categories. Given a finite semisimple tensor 1-category $\mathcal{C}$, $Hom_{\mathbf{TC}^{rsep}}(\mathbf{Vect}_{\mathds{k}},\mathcal{C})$ is by definition the 2-category of separable right $\mathcal{C}$-module 1-categories, whence the above 3-functor factors through the full sub-2-category $\mathbf{CSS2C}$. We denote this linear 3-functor by $$\mathbf{Mod}:\mathbf{TC}^{rsep}\rightarrow \mathbf{CSS2C}.$$

We will prove that $\mathbf{Mod}$ is an equivalence by showing that it is essentially surjective and induces equivalences on $Hom$-2-categories. Essential surjectivity follows from theorem \ref{thm:semisimplemodule}. Now, let $\mathcal{C}$, $\mathcal{D}$ be two finite semisimple tensor 1-categories. It remains to prove that the linear 2-functor $$T:Hom_{\mathbf{TC}^{rsep}}(\mathcal{C},\mathcal{D})\rightarrow Fun(\mathbf{Mod}(\mathcal{C}),\mathbf{Mod}(\mathcal{D}))$$ induced by $\mathbf{Mod}(-)$ is an equivalence. By construction, this 2-functor sends the $\mathcal{C}$-$\mathcal{D}$-bimodule 1-category $\mathcal{M}$, which is separable as a right module 1-category, to the 1-functor $$\begin{tabular}{c c c}
    $\mathbf{Mod}(\mathcal{C})$ & $\rightarrow$ & $\mathbf{Mod}(\mathcal{D})$ \\
     $\mathcal{N}$ & $\mapsto$ & $\mathcal{N}\boxtimes_{\mathcal{C}} \mathcal{M}.$
\end{tabular}$$ A pseudo-inverse to $T$ is given by the linear 2-functor $$R:Fun(\mathbf{Mod}(\mathcal{C}),\mathbf{Mod}(\mathcal{D}))\rightarrow Hom_{\mathbf{TC}^{rsep}}(\mathcal{C},\mathcal{D})$$ defined by sending the 2-functor $F:\mathbf{Mod}(\mathcal{C})\rightarrow\mathbf{Mod}(\mathcal{D})$ to $F(\mathcal{C})$. Note that $F(\mathcal{C})$ inherits a left action by $\mathcal{C}$ from the canonical left action of $\mathcal{C}$ on itself. The functoriality of $F$ ensures that it is compatible with the given right action by $\mathcal{D}$, so that $F(\mathcal{C})$ is a $\mathcal{C}$-$\mathcal{D}$-bimodule 1-category, which is separable as a right $\mathcal{D}$-module 1-category. Then, it is not hard to check directly that $R \circ T$ is equivalent to the identity 2-functor on $\mathbf{TC}^{rsep}$. Finally, similarly to what is done in the proof of theorem 2.2.2 of \cite{D1}, one can use the 3-universal property of the Cauchy completion (see definition 1.2.1 of \cite{D1}) to show that $T\circ R$ is equivalent to the identity 2-functor on $\mathbf{CSS2C}$.
\end{proof}

\begin{Corollary}
Over a perfect field $\mathds{k}$, the equivalence of theorem \ref{thm:equivalence} restricts to an equivalence between the full sub-3-categories on the separable tensor 1-categories and on the locally separable compact semisimple 2-categories.
\end{Corollary}
\begin{proof}
This is a consequence of theorems \ref{thm:sepmultsep2cat} and \ref{thm:sep2catsepmult}
\end{proof}

\begin{Corollary}
If $\mathds{k}$ is an algebraically closed field or a real closed field, the equivalence of theorem \ref{thm:equivalence} becomes an equivalence between the 3-category of multifusion 1-categories and that of finite semisimple 2-categories.
\end{Corollary}
\begin{proof}
This follows from theorem \ref{thm:modulefinitesemisimple}.
\end{proof}

\subsection{The Yoneda Embedding}

In \cite{DR}, the authors prove that the absolute Yoneda embedding of a finite semisimple 2-category over an algebraic closed field of characteristic zero is an equivalence (proposition 1.4.11). Assuming that $\mathds{k}$ is a perfect field, their proof generalizes to locally separable compact semisimple 2-categories. We present here a slightly different argument using theorem \ref{thm:equivalence}.

\begin{Proposition}\label{prop:Yoneda}
Le $\mathds{k}$ be a perfect field, and $\mathfrak{C}$ be a locally separable compact semisimple 2-category. The Yoneda embedding $$\yo:\mathfrak{C}\rightarrow Fun(\mathfrak{C}^{1op}, \mathbf{2Vect}_{\mathds{k}})$$ is an equivalence.
\end{Proposition}
\begin{proof}
Thanks to theorem \ref{thm:semisimplemodule}, we may assume that $\mathfrak{C}=\mathbf{Mod}(\mathcal{C})$ for some separable tensor 1-category $\mathcal{C}$. Moreover, we have $$\mathbf{Mod}(\mathcal{C}^{\otimes op})\simeq Cau(\mathrm{B}(\mathcal{C}^{\otimes op}))\simeq Cau(\mathrm{B}(\mathcal{C})^{1 op})\simeq Cau(\mathrm{B}(\mathcal{C}))^{1 op}\simeq \mathbf{Mod}(\mathcal{C})^{1op}.$$ Under this equivalence, the finite semisimple right $\mathcal{C}$-module 1-category $\mathrm{Mod}_{\mathcal{C}^{\otimes op}}(A)$ is sent to $\mathrm{Mod}_{\mathcal{C}}(A)$.

Now, by theorem \ref{thm:equivalence} and the above equivalence, $Fun(\mathfrak{C}^{1op}, \mathbf{2Vect}_{\mathds{k}})$ is equivalent to the 2-category $\mathbf{LMod}(\mathcal{C}^{\otimes op})$ of finite semisimple left $\mathcal{C}^{\otimes op}$-module 1-categories via evaluation at $\mathcal{C}$. More precisely, given a 2-functor $F:\mathfrak{C}^{1op}\rightarrow \mathbf{2Vect}_{\mathds{k}}$, we view $F(\mathcal{C})$ as a left $\mathcal{C}^{\otimes op}$-module 1-category by precomposition using the left action of $\mathcal{C}$ on itself. Under these identifications, the Yoneda embedding is given by $$\mathcal{M}\mapsto Hom_{\mathcal{C}}(\mathcal{C},\mathcal{M}).$$ Observe that the right hand-side is equivalent to $\mathcal{M}$ viewed as a left $\mathcal{C}^{\otimes op}$-module 1-category. This 2-functor is therefore manifestly an equivalence, so the proof of the result is complete.
\end{proof}

\begin{Remark}
We have to assume that $\mathds{k}$ is perfect in the statement of proposition \ref{prop:Yoneda}. Namely, if $\mathds{k}$ is arbitrary, the 2-functor $\yo$ may not make sense as $Hom_{\mathfrak{C}}(C,D)$ is a priori only a finite semisimple $\mathds{k}$-linear 1-category, and not a separable $\mathds{k}$-linear 1-category. Explicitly, let us set $\mathds{k}=\mathbb{F}_p(x)$, and define $\mathds{k'}=\mathbb{F}_p(\sqrt[p]{x})$. If we take $\mathfrak{C}=\mathbf{2Vect}_{\mathds{k}}$, and $C=\mathbf{Vect}_{\mathds{k}}$, $D=\mathbf{Vect}_{\mathds{k'}}$, then $Hom_{\mathfrak{C}}(C,D) = \mathbf{Vect}_{\mathds{k'}}$, which is not a separable $\mathds{k}$-linear 1-category.

Similarly, assuming that $\mathds{k}$ is perfect alone is not enough. Namely, theorem \ref{thm:equivalence} identifies $Fun(\mathbf{Mod}(\mathcal{C})^{1op}, \mathbf{2Vect}_{\mathds{k}})$ with the 2-category of left $\mathcal{C}^{\otimes op}$-module 1-categories whose underlying 1-category is separable. Note that, in this case, the 2-functor $\yo$ does indeed exist. Taking $\mathds{k}=\mathbb{F}_p$, and $\mathcal{C}$ to be the finite semisimple tensor 1-category of $\mathbb{Z}/p$-graded finite dimensional $\mathds{k}$-vector spaces, we see that $\mathbf{Vect}_{\mathds{k}}$ is a left $\mathcal{C}^{\otimes op}$-module 1-category whose underlying 1-category is separable. However, it cannot be in the image of $\yo$ as it is not separable as a left $\mathcal{C}^{\otimes op}$-module 1-category.
\end{Remark}

\begin{Corollary}
Let $\mathds{k}$ be a perfect field, $\mathfrak{C}$, $\mathfrak{D}$ be two compact semisimple 2-categories, and $P:\mathfrak{D}^{1op}\times\mathfrak{C}\rightarrow \mathbf{2Vect}_{\mathds{k}}$ be a bilinear 2-functor. If $\mathfrak{D}$ is locally separable, there exists a linear 2-functor $F:\mathfrak{C}\rightarrow\mathfrak{D}$ such that $P$ and $Hom_{\mathfrak{D}}(-,F(-))$ are 2-naturally equivalent.
\end{Corollary}
\begin{proof}
By hypothesis, $P$ defines a linear 2-functor $Q:\mathfrak{C}\rightarrow Fun(\mathfrak{D}^{1op},\mathbf{2Vect}_{\mathds{k}})$. Composing $Q$ with a pseudo-inverse of $\yo$, we obtain a linear 2-functor $F:\mathfrak{C}\rightarrow \mathfrak{D}$. For every $C$ in $\mathfrak{C}$, and $D$ in $\mathfrak{D}$, there are natural equivalences \begin{align*}Hom_{\mathfrak{D}}(D,F(C))&\simeq Hom_{Fun(\mathfrak{D}^{1op},\mathbf{2Vect}_{\mathds{k}})}(\yo(D),\yo(F(C)))\\&\simeq Hom_{Fun(\mathfrak{D}^{1op},\mathbf{2Vect}_{\mathds{k}})}(\yo(D),Q(C))\\&\simeq (Q(C))(D) = P(D,C).\end{align*} This proves the result.
\end{proof}

\subsection{Compact Semisimple Tensor 2-Categories}

The following definitions are straightfoward categorifications of the definitions of tensor 1-category and of (multi)fusion 1-category recalled in \ref{sub:sepmod}.

\begin{Definition}
A tensor 2-category is a rigid monoidal linear 2-category.
\end{Definition}

\begin{Definition}
Over an algebraically closed field or a real closed field, a multifusion 2-category is a rigid monoidal finite semisimple 2-category. A fusion 2-category is a multifusion 2-category whose monoidal unit is a simple object.
\end{Definition}

The results of section 2 of \cite{D2} also apply to compact semisimple tensor 2-categories. We highlight the generalized versions of proposition 2.4.3 and proposition 2.4.7. We omit the proofs as they are virtually identical to those given in \cite{D2}.

\begin{Proposition}
Let $\mathcal{C}$ be a braided finite semisimple tensor 1-category. Then, $\mathbf{Mod}(\mathcal{C})$ is a compact semisimple tensor 2-category with monoidal structure given by the balanced Deligne tensor product $\boxtimes_{\mathcal{C}}$.
\end{Proposition}

\begin{Proposition}
There is an equivalence between the category of connected compact semisimple tensor 2-categories and equivalence classes of monoidal linear 2-functors and the category of finite semisimple braided tensor 1-categories with simple monoidal unit and equivalence classes of braided tensor functors.
\end{Proposition}

\begin{Example}
Let us work over $\mathbb{R}$, and examine the fusion rule of the fusion 2-categories $\mathbf{Mod}(\mathbf{Vect}_{\mathbb{R}})$ and $\mathbf{Mod}(\mathbf{Vect}_{\mathbb{C}})$. From example \ref{ex:vectRC}, we know that the latter has only one simple object, which is the monoidal unit, so the fusion rule is completely determined. We also know that the former has three equivalence classes of simple objects given by $\mathbf{Vect}_{\mathbb{R}}$, $\mathbf{Vect}_{\mathbb{C}}$, and $\mathbf{Vect}_{\mathbb{H}}$. Clearly, $\mathbf{Vect}_{\mathbb{R}}$ is the monoidal unit. Now, the remaining products are given by: $$\mathbf{Vect}_{\mathbb{C}}\boxtimes_{\mathbf{Vect}_{\mathbb{R}}}\mathbf{Vect}_{\mathbb{C}}\simeq \mathbf{Vect}_{\mathbb{C}}\boxplus \mathbf{Vect}_{\mathbb{C}},$$ $$\mathbf{Vect}_{\mathbb{C}}\boxtimes_{\mathbf{Vect}_{\mathbb{R}}}\mathbf{Vect}_{\mathbb{H}}\simeq\mathbf{Vect}_{\mathbb{H}}\boxtimes_{\mathbf{Vect}_{\mathbb{R}}}\mathbf{Vect}_{\mathbb{C}}\simeq \mathbf{Vect}_{\mathbb{C}} ,$$ $$\mathbf{Vect}_{\mathbb{H}}\boxtimes_{\mathbf{Vect}_{\mathbb{R}}}\mathbf{Vect}_{\mathbb{H}}\simeq \mathbf{Vect}_{\mathbb{R}}.$$
\end{Example}

\begin{Example}
Let us examine the structure of the compact semisimple tensor 2-category $\mathbf{Mod}(\mathbf{Vect}_{\mathbb{F}_p})$ for some fixed prime $p$. Equivalence classes of simple objects are classified by finite separable division algebras over $\mathbb{F}_p$. But, by the Wedderburn theorem, every such algebra is isomorphic to $\mathbb{F}_{p^q}$ for some positive integer $q$. This means that a set of representatives for the equivalence classes of simple objects of $\mathbf{Mod}(\mathbf{Vect}_{\mathbb{F}_p})$ is given by $\mathbf{Vect}_{\mathbb{F}_{p^q}}$ for all positive integer $q$. For any two positive integers $q\geq r$, we have $$\mathbf{Vect}_{\mathbb{F}_{p^q}}\boxtimes_{\mathbf{Vect}_{\mathbb{F}_p}} \mathbf{Vect}_{\mathbb{F}_{p^r}}\simeq \boxplus_{i=1}^{r}\mathbf{Vect}_{\mathbb{F}_{p^q}}.$$ Namely, $\mathbb{F}_{p^q}\otimes \mathbb{F}_{p^r}=\oplus_{i=1}^r \mathbb{F}_{p^q}$. A similar identity holds if $r\leq q$.
\end{Example}

\begin{Example}
Let $p$ be a prime number, and $\mathds{k}$ be an algebraically closed field such that $char(\mathds{k})\neq p$. We work with $\mathbf{Vect}(\mathbb{Z}/p)$, the fusion 1-category of finite dimensional $\mathbb{Z}/p$-graded $\mathds{k}$-vector spaces. A braiding $b$ on $\mathbf{Vect}(\mathbb{Z}/p)$ is completely determined by $$b_{\mathds{k}_1,\mathds{k}_1}:\mathds{k}_1\otimes\mathds{k}_1\rightarrow \mathds{k}_1\otimes\mathds{k}_1,$$ where $\mathds{k}_1$ is the one dimensional $\mathds{k}$-vector space concentrated in degree $1\in \mathbb{Z}/p$. Furthermore, the braiding axioms imply that $b_{\mathds{k}_1,\mathds{k}_1}$ has to be a $p$-th root of unity $\zeta$. Conversely, any choice of a $p$-th root of unity $\zeta$ describes a braiding on $\mathbf{Vect}(\mathbb{Z}/p)$ via the assignment $b_{\mathds{k}_1,\mathds{k}_1}=\zeta$.

On one hand, if $p=char(\mathds{k})$, then we have seen in example \ref{ex:FpalgZp} that the finite semisimple 2-category $\mathbf{Mod}(\mathbf{Vect}(\mathbb{Z}/p))$ has a unique equivalence class of simple objects given by $\mathbf{Vect}(\mathbb{Z}/p)$. Further, in this case, $1$ is the unique $p$-th root of unity, so there is a unique braiding $triv$ on $\mathbf{Vect}(\mathbb{Z}/p)$. We write $\mathcal{C}$ for the corresponding braided fusion 1-category. The fusion rule of $\mathbf{Mod}(\mathcal{C})$ is then completely described by $\mathbf{Vect}(\mathbb{Z}/p)\boxtimes_{\mathcal{C}}\mathbf{Vect}(\mathbb{Z}/p)\simeq \mathbf{Vect}(\mathbb{Z}/p)$, as $\mathbf{Vect}(\mathbb{Z}/p)$ is the monoidal unit.

On the other hand, if $p\neq char(\mathds{k})$, then we have seen how to describe the equivalence classes of simple of $\mathbf{Mod}(\mathbf{Vect}(\mathbb{Z}/p))$ in example \ref{ex:separablevectG}. In our current situation, we have two equivalence classes of simple objects, $\mathbf{Vect}(\mathbb{Z}/p)$ with the canonical right action (corresponding to the zero subgroup with the trivial cohomology class), and $\mathbf{Vect}$ with the trivial right action (corresponding to the group $\mathbb{Z}/p$ with the trivial cohomology class).

Let us denote by $\mathcal{C}$ the braided fusion 1-category obtained by giving $\mathbf{Vect}(\mathbb{Z}/p)$ the braiding specified by $1$ in $\mathds{k}$, and by $\mathcal{D}$ the braided fusion 1-category obtained by endowing $\mathbf{Vect}(\mathbb{Z}/p)$ with the braiding $\beta$ specified by a fixed primitive $p$-th root of unity $\zeta$ in $\mathds{k}$. We now describe the fusion rules of the associated compact semisimple tensor 2-categories. In either case, the simple object $\mathbf{Vect}(\mathbb{Z}/p)$ is the monoidal unit, so we only have to describe the product of $\mathbf{Vect}$ by $\mathbf{Vect}$. In $\mathbf{Mod}(\mathcal{C})$, we have $$\mathbf{Vect}\boxtimes_{\mathcal{C}}\mathbf{Vect}\simeq \boxplus_{i=1}^p \mathbf{Vect},$$ and in $\mathbf{Mod}(\mathcal{D})$, we have $$\mathbf{Vect}\boxtimes_{\mathcal{D}}\mathbf{Vect}\simeq \mathbf{Vect}(\mathbb{Z}/p).$$ This can be seen by generalizing the argument of section 2.5 of \cite{D2} to the characteristic $p$ case. Namely, the argument only relies on the field being algebraically closed.
\end{Example}

\appendix

\section{Weak Based Rings}
\setcounter{subsection}{1}

The goal of this part of the appendix is to state a version of proposition 3.4.6 of \cite{EGNO} that applies to weak based rings. As indicated in exercise 3.8.1, such a statement holds. We provide a proof for completeness. Further, we also prove lemma \ref{lem:Z+homfinite}, as it features prominently in the proof of our main theorem, and we have not been able to locate a proof in the literature. We begin by recalling the necessary definitions from \cite{EGNO} section 3.1 and section 3.8.

\begin{Definition}
A $\mathbb{Z}_+$-ring is a ring $A$, which is free as $\mathbb{Z}$-module, together with basis $\{b_i\}_{i\in I}$ such that $1=\sum_{i}a^ib_i$ for non-negative integers $a^i$ and, for every $i$, $j$ we have $b_ib_j = \sum_{k}c_{ij}^k b_k$ for non-negative integers $c_{ij}^k$.
\end{Definition}

Given a $\mathbb{Z}_+$-ring $A$, we write $I_0$ for the finite subset of $I$ corresponding to the non-zero $a^i$'s. We define $\tau:A\rightarrow \mathbb{Z}$ to be the group homomorphism given by $$\tau(b_i)=\begin{cases}1 \textrm{ if } i\in I_0\\0 \textrm{ otherwise.} \end{cases}$$

\begin{Definition}
A $\mathbb{Z}_+$-ring $A$ with basis $\{b_i\}_{i\in I}$ is called a weak based ring if there exists an involution $i\mapsto i^*$ of $I$ such that the induced map on $A$ is an anti-involution, and $$\tau(b_ib_j)=\begin{cases}t_i>0 \textrm{ if } i=j^*\\0  \textrm{ if } i\neq j^*. \end{cases}$$
If $t_i=1$ for all $i$, then we say that $A$ is a based ring.
\end{Definition}

The following example also appears in \cite{EG}.

\begin{Example}
The Grothendieck ring $Gr(\mathcal{C})$ of a finite semisimple tensor 1-category $\mathcal{C}$ with unit $I$ is a weak based ring with basis given by the equivalence classes of simple objects and involution induced by the dual functor. Namely, the usual proof shows that right and left duals agree up to a non-canonical isomorphism. Namely, for any simple object $C$ of $\mathcal{C}$, we have $$Hom_{\mathcal{C}}(C^*,{^*C})\cong Hom_{\mathcal{C}}(C\otimes C^*,I)\cong Hom_{\mathcal{C}}(I, C\otimes C^*)\cong Hom_{\mathcal{C}}(C,C).$$ For any simple objects $C$, $D$ in $\mathcal{C}$, there is an isomorphism $$Hom_{\mathcal{C}}(C^*\otimes D,I)\cong Hom_{\mathcal{C}}(D,C),$$ and, as the right hand-side is non-zero if and only if $C$ and $D$ are isomorphic, this shows that $Gr(\mathcal{C})$ is a weak based ring. Note that if $\mathds{k}$ is algebraically closed, then the Grothendieck ring of $\mathcal{C}$ is a based ring.
\end{Example}

\begin{Definition}
Let $A$, $C$ be two $\mathbb{Z}_+$-rings with bases $\{b_i\}_{i\in I}$, and $\{d_j\}_{j\in J}$ respectively. A homomorphism of $\mathbb{Z}_+$-rings $f:A\rightarrow C$ is a homomorphism of rings such that for every $i$, $f(b_i)=\sum_{j\in J}g^j_id_j$ for some non-negative integers $g^j_i$.
\end{Definition}

\begin{Definition}
Let $A$, $C$ be two weak based $\mathbb{Z}_+$-rings. A homomorphism of weak based $\mathbb{Z}_+$-rings $f:A\rightarrow C$ is a homomorphism of $\mathbb{Z}_+$-rings which preserves the involution.
\end{Definition}

\begin{Example}
A tensor functor $F:\mathcal{C}\rightarrow \mathcal{D}$ between two finite semisimple tensor 1-categories induces a homomorphism of weak based $\mathbb{Z}_+$-rings $Gr(F):Gr(\mathcal{C})\rightarrow Gr(\mathcal{D})$ by sending the equivalence class of the simple object $C$ to $F(C)$.
\end{Example}

\begin{Lemma}\label{lem:Z+homfinite}
The set of homomorphisms of weak based $\mathbb{Z}_+$-rings between two weak based $\mathbb{Z}_+$-rings of finite rank is finite.
\end{Lemma}
\begin{proof}
Let $A$, $C$ be two $\mathbb{Z}_+$-rings with bases $\{b_i\}_{i\in I}$, and $\{d_j\}_{j\in J}$ respectively. For any $j,k$ in $J$, we have $d_jd_k = \sum_{l\in J} c^l_{jk}d_l$ for some non-negative integers $c^l_{jk}$. We write $f:A\rightarrow C$ for some homomorphism $\mathbb{Z}_+$-rings. Let us define $b:=\sum_{i\in I}b_i$. We have $b^2 = \sum_{j\in I}n^ib_i$, for some non-negative integers $n_i$, and we write $N:=\mathrm{max}\{n_i\}$. Further, we also have $f(b_i)=\sum_{j\in J}g^j_id_j$ for some non-negative integers $g^j_i$. Summing over $i$, we write $f(b)=\sum_{j\in J}g^jd_j$ for some non-negative integers $g^j$. We now compute $f(b^2) = \sum_{j\in J}e^jd_j$ in two different ways. Firstly, we have $$f(b^2) = f(\sum_{i\in I}n^ib_i) = \sum_{i\in I, j\in J}g^j_i n^i d_j,$$ from which we find $\sum_{j\in J}e^j\leq N\sum_j g^j$. Secondly, we have $$f(b^2) = (\sum_{j\in J}g^jd_j)(\sum_{k\in J}g^kd_k)=\sum_{j,k,l}c^l_{jk}g^jg^kd_l.$$ As $f$ is a homomorphism of weak based $\mathbb{Z}_+$-rings, and $b = b^*$, we have $f(b)=f(b)^*$. This together with the above equation implies that $\sum_{j\in J}e^j\geq \sum_j (g^j)^2$. Putting both inequalities together, we have $\sum_j (g^j)^2\leq N\sum_j g^j$, which forces $g^j\leq |J| N$ for all $j$, whence the proof is complete.
\end{proof}

We now recall two definitions from section 3.4 of \cite{EGNO}.

\begin{Definition}
Let $A$ be a $\mathbb{Z}_+$-ring with basis $\{b_i\}_{i\in I}$. A $\mathbb{Z}_+$-module over $A$ is a right $A$-module $M$, which is free as $\mathbb{Z}$-module, together with a basis $\{m_l\}_{l\in L}$ such that for all $i$, $l$, we have $m_lb_i=\sum_k a^k_{il}m_k$ for some non-negative integers $a^k_{il}$.
\end{Definition}

\begin{Definition}
A $\mathbb{Z}_+$-module $M$ over $A$ is called indecomposable if it cannot be written as the direct sum of two non-zero $\mathbb{Z}_+$-modules over $A$. A $\mathbb{Z}_+$-module $M$ over $A$ is called irreducible if it no proper non-empty subset of the basis generates a proper submodule over $A$.
\end{Definition}

\begin{Example}
Let $\mathcal{C}$ be a finite semisimple tensor 1-category, and $\mathcal{M}$ a finite semisimple right $\mathcal{C}$-module 1-category. Then, the action of $Gr(\mathcal{C})$ on $Gr(\mathcal{M})$ the Grothendieck group of $\mathcal{M}$, endows $Gr(\mathcal{M})$ with the structure of $\mathbb{Z}_+$-module over $Gr(\mathcal{C})$. It is easy to see that $\mathcal{M}$ is indecomposable as a right $\mathcal{C}$-module 1-category if and only if $Gr(\mathcal{M})$ is indecomposable as a $\mathbb{Z}_+$-module over $Gr(\mathcal{C})$.
\end{Example}

The following lemma generalizes proposition 7.7.2 of \cite{EGNO}.

\begin{Lemma}
Let $\mathcal{M}$ be a finite semisimple right $\mathcal{C}$-module 1-category. Then, the $\mathbb{Z}_+$-module $Gr(\mathcal{M})$ over $Gr(\mathcal{C})$ is irreducible.
\end{Lemma}
\begin{proof}
Let $N$ be a proper $\mathbb{Z}_+$-submodule of $Gr(\mathcal{M})$ stable under $Gr(\mathcal{C})$. These conditions ensures that there exists a proper finite semisimple right $\mathcal{C}$-submodule 1-category $\mathcal{N}$ of $\mathcal{M}$ such that $Gr(\mathcal{N}) = N$. Similarly, if we write $N^{\bot}$ for the proper $\mathbb{Z}_+$-submodule of $Gr(\mathcal{M})$ spanned by the basis vector of $Gr(\mathcal{M})$ not in $N$, there exists a proper finite semisimple subcategory $\mathcal{N}^{\bot}$ of $\mathcal{M}$ such that $Gr(\mathcal{N}^{\bot})=N^{\bot}$. We claim that $\mathcal{N}^{\bot}$ is stable under the right action of $\mathcal{C}$. Provided this is true, we have exhibited $\mathcal{M}$ as a direct sum of the two non-trivial finite semisimple right $\mathcal{C}$-module 1-categories $\mathcal{N}$ and $\mathcal{N}^{\bot}$, a contradiction.

Let us now prove the claim. Let $C$ in $\mathcal{C}$, $M$ in $\mathcal{N}^{\bot}$, and $N$ in $\mathcal{N}$ be arbitrary objects, there is an isomorphism $$Hom_{\mathcal{M}}(M\otimes C,N)\cong Hom_{\mathcal{M}}(M,N\otimes {C^*}).$$ Thus, we see that $\mathcal{N}$ is stable under the action of $\mathcal{C}$ if and only if $\mathcal{N}^{\bot}$ is, which proves the claim.
\end{proof}

The next proposition is proposition 3.4.6 of \cite{EGNO} stated in the setting of weak based rings. We follow their proof very closely.

\begin{Proposition}\label{prop:weak346}
Let $A$ be a weak based ring of finite rank. Then there are only finitely many equivalences classes of irreducible $\mathbb{Z}_+$-modules over $A$.
\end{Proposition}
\begin{proof}
As $A$ has finite rank, so will any irreducible $\mathbb{Z}_+$-modules over $A$. Namely, let $M$ be an irreducible $\mathbb{Z}_+$-modules over $A$ with basis $\{m_l\}_{l\in L}$. Given any $l$ in $L$ and $i$ in $I$, we write $m_lb_i=\sum_{k\in L}f_{li}^km_k$ for some non-negative integers $f_{li}^k$. If we let $b:=\sum_{i\in I}b_i$, we have $m_lb=\sum_{k\in L}f_{l}^km_k$ with $f_l^k=\sum_{i\in I}f_{li}^k$, which is positive as $A$ has a unit. We let $\langle l\rangle:=\{k\in L|f_l^k\neq 0\}$, which is finite, and claim that the $\mathbb{Z}_{+}$-submodule of $M$ generated by $m_k$ with $k\in\langle l\rangle$ is stable under the action of $A$. Namely, for any $k$ in $\langle l\rangle$, $m_k$ appears in $m_lb$. On the other hand, $(m_lb)b=m_l(b^2)=\sum_{i\in I, k\in\langle l\rangle}n_if_{il}^k m_k$. This proves the claim. Further, as $M$ is irreducible, we have $L=\langle l\rangle$.

In $A$, we have $b^2 = \sum_{i \in I}n_ib_i$ for some non-negative integers $n_i$, and we define $N:=\mathrm{max}\{n_i\}$, which is a positive integer independent of $M$. We also set $f_l:=\sum_{k\in L}f_{l}^k$. Now, pick $l_0$ in $L$ such that $f_{l_0}$ is minimal amongst the $f_l$'s. We now compute $m_{l_0}b^2 = \sum_{k\in L}c_{k}m_k$ in two different ways. Firstly, we have $$(m_{l_0}b)b = \sum_{j,k\in L}f_{l_0}^jf_{j}^km_k,$$ from which we deduce $\sum_{k\in L}c_{k}\geq f_{l_0}^2$. Secondly, we have $$m_{l_0}(b^2) = \sum_{i\in I, k\in L} n_i f_{l_0i}^k m_k,$$ from which we get $Nf_{l_0}\geq \sum_{k\in L}c_{k}$. Putting everything together, we find $Nf_{l_0}\geq f_{l_0}^2$, and so $N\geq l_0$. As $L=\langle l_0\rangle$ because $M$ is irreducible, this bounds the rank of $M$. Moreover, we also get $N^2\geq \sum_{k\in L}c_k = \sum_{j,k\in L}f_{l_0}^jf_{j}^k$. As $f_{l_0}^j$ is positive for every $j$, there are only finitely many non-negative integers $f_{j}^k$ satisfying this inequality. This completes the proof of the proposition.
\end{proof}

\bibliography{bibliography.bib}

\end{document}